\documentclass[10pt]{scrartcl}

\usepackage{amsmath, amssymb, amsthm, color, url, fullpage, supertabular}
\usepackage[colorlinks, linkcolor=black, urlcolor=blue, citecolor=black]{hyperref}
\usepackage[neverdecrease]{paralist}
\usepackage{mleftright} \mleftright 

\usepackage{natbib}
\setcitestyle{round}
\newcommand*{\arXiv}[1]{\bgroup\color{blue}\href{https://arxiv.org/abs/#1}{arXiv:#1}\egroup}
\newcommand*{\doi}[1]{\bgroup\color{blue}\href{https://dx.doi.org/#1}{doi:#1}\egroup}
\newcommand*{\email}[1]{\bgroup\color{blue}\href{mailto:#1}{#1}\egroup}
\renewcommand*{\url}[1]{\bgroup\color{blue}\href{#1}{#1}\egroup}
\newcommand{\todo}[1]{\bgroup\color{red}\bfseries#1\egroup}

\newcommand*{\ppara}[1]{\noindent\textbf{\textsf{#1}}\,\,}

\usepackage{mathtools}
\newcommand*{\defeq}{\coloneqq}
\newcommand*{\qefed}{\eqqcolon}

\newcommand*{\dhh}{d_{\textup{H}}}
\newcommand*{\Misfit}{\Phi}
\newcommand*{\Naturals}{\mathbb{N}}
\newcommand*{\Normal}{\mathcal{N}}
\newcommand*{\ProbSpace}{\Omega}
\newcommand*{\ProbSimplex}[1]{\mathcal{M}_{1}(#1)}
\newcommand*{\quark}{\setbox0\hbox{$x$}\hbox to\wd0{\hss$\cdot$\hss}}
\newcommand*{\rd}{\mathrm{d}}
\newcommand*{\Reals}{\mathbb{R}}
\newcommand*{\transpose}{\mathtt{T}}
\newcommand*{\UU}{\mathcal{U}}
\newcommand*{\YY}{\mathcal{Y}}
\newcommand*{\VV}{\mathcal{V}}

\newcommand{\EE}{\mathbb{E}}
\newcommand{\Var}{\mathbb{V}}
\newcommand{\R}{\mathbb{R}}

\newcommand{\marginal}{\textup{M}}
\newcommand{\sample}{\textup{S}}

\newcommand*{\odeflow}{F}
\newcommand*{\numflow}{\Psi}
\newcommand*{\tstep}{\tau}

\newcommand*{\norm}[1]{\Vert #1 \Vert}

\newcommand*{\absval}[1]{\vert #1 \vert}
\newcommand*{\bignorm}[1]{\bigl\Vert #1 \bigr\Vert}

\newcommand*{\bigabsval}[1]{\bigl\vert #1 \bigr\vert}
\newcommand*{\Norm}[1]{\left\Vert #1 \right\Vert}

\newcommand*{\Innerprod}[2]{\left\langle #1 , #2 \right\rangle}
\newcommand*{\Absval}[1]{\left\vert #1 \right\vert}

\newtheorem{theorem}{Theorem}[section]
\newtheorem{corollary}[theorem]{Corollary}
\newtheorem{lemma}[theorem]{Lemma}
\newtheorem{proposition}[theorem]{Proposition}
\theoremstyle{definition}

\newtheorem{assumption}[theorem]{Assumption}
\newtheorem{remark}[theorem]{Remark}

\numberwithin{equation}{section}

\begin{document}

\title{Random forward models and log-likelihoods in Bayesian inverse problems}

\author{%
	H.\ C.\ Lie\footnote{Institute of Mathematics, Freie Universit{\"a}t Berlin, Arnimallee 6, 14195 Berlin, Germany, \email{hlie@math.fu-berlin.de}} %
	\and %
	T.\ J.\ Sullivan\footnote{Institute of Mathematics, Freie Universit{\"a}t Berlin, Arnimallee 6, 14195 Berlin, Germany, \email{t.j.sullivan@fu-berlin.de}} \footnote{Zuse Institute Berlin, Takustra{\ss}e 7, 14195 Berlin, Germany, \email{sullivan@zib.de}}%
	\and %
	A.\ L.\ Teckentrup\footnote{School of Mathematics, University of Edinburgh, UK, \email{a.teckentrup@ed.ac.uk}} \footnote{The Alan Turing Institute, 96 Euston Road, London, NW1 2DB, UK}
}

\date{\today}

\maketitle

\begin{abstract}
	\ppara{Abstract:}
	We consider the use of randomised forward models and log-likelihoods within the Bayesian approach to inverse problems.
	Such random approximations to the exact forward model or log-likelihood arise naturally when a computationally expensive model is approximated using a cheaper stochastic surrogate, as in Gaussian process emulation (kriging), or in the field of probabilistic numerical methods.
	We show that the Hellinger distance between the exact and approximate Bayesian posteriors is bounded by moments of the difference between the true and approximate log-likelihoods.
	Example applications of these stability results are given for randomised misfit models in large data applications and the probabilistic solution of ordinary differential equations.
	
	\smallskip
	
	\ppara{Keywords:} Bayesian inverse problem, random likelihood, surrogate model, posterior consistency, uncertainty quantification, randomised misfit, probabilistic numerics.
	
	\smallskip
	
	\ppara{2010 Mathematics Subject Classification:} 
	62F15, 62G08, 65C99, 65D05, 65D30, 65J22, 68W20.
\end{abstract}	

\section{Introduction}
\label{sec:introduction}

Inverse problems are ubiquitous in the applied sciences and in recent years renewed attention has been paid to their mathematical and statistical foundations \citep{EvansStark:2002, KaipioSomersalo:2005, Stuart:2010}.
Questions of well-posedness --- i.e.\ the existence, uniqueness, and stability of solutions --- have been of particular interest for infinite-dimensional/non-parametric inverse problems because of the need to ensure stable and discretisation-independent inferences \citep{LassasSiltanen:2004} and develop algorithms that scale well with respect to high discretisation dimension \citep{CotterRobertsStuartWhite:2013}.

This paper considers the stability of the posterior distribution in a Bayesian inverse problem (BIP) when an accurate but computationally intractable forward model or likelihood is replaced by a random surrogate or emulator.
Such stochastic surrogates arise often in practice.
For example, an expensive forward model such as the solution of a PDE may replaced by a kriging/Gaussian process (GP) model \citep{StuartTeckentrup:2017}.
In the realm of ``big data'' a residual vector of prohibitively high dimension may be randomly subsampled or orthogonally projected onto a randomly-chosen low-dimensional subspace \citep{LeMyersBuiThanhNguyen:2017, NemirovskiJuditskyLanShapiro:2008}.
In the field of probabilistic numerical methods \citep{HennigOsborneGirolami:2015}, a deterministic dynamical system may be solved stochastically, with the stochasticity representing epistemic uncertainty about the behaviour of the system below the temporal or spatial grid scale \citep{Conrad:2016, LieStuartSullivan:2017}.

In each of the above-mentioned settings, the stochasticity in the forward model propagates to associated inverse problems, so that the Bayesian posterior becomes a \emph{random measure}, $\mu_{N}^{\sample}$, which we define precisely in \eqref{eq:random_posterior}.
Alternatively, one may choose to average over the randomness to obtain a \emph{marginal posterior}, $\mu_{N}^{\marginal}$, which we define precisely in \eqref{eq:marginal_posterior}.
It is natural to ask in which sense the approximate posterior (either the random or the marginal version) is close to the ideal posterior of interest, $\mu$.


In earlier work, \citet{StuartTeckentrup:2017} examined the case in which the random surrogate was a GP.
More precisely, the object subjected to GP emulation was either the forward model (i.e.\ the parameter-to-observation map) or the negative log-likelihood.
The prior GP was assumed to be continuous, and was then conditioned upon finitely many observations (i.e.\ pointwise evaluations) of the parameter-to-observation map or negative log-likelihood as appropriate.
That paper provided error bounds on the Hellinger distance between the BIP's exact posterior distribution and various approximations based on the GP emulator, namely approximations based on the mean of the predictive (i.e.\ conditioned) GP, as well as approximations based on the full GP emulator.
Those results showed that the Hellinger distance between the exact BIP posterior and its approximations can be bounded by moments of the error in the emulator.

In this paper, we extend the analysis of \citet{StuartTeckentrup:2017} to consider more general (i.e.\ non-Gaussian) random approximations to forward models and log-likelihoods, and quantify the impact upon the posterior measure in a BIP.
After establishing some notation in Section~\ref{sec:setup}, we state the main approximation theorems in Section~\ref{sec:general}.
Section~\ref{sec:random-misfit} gives an application of the general theory to random misfit models, in which high-dimensional data are rendered tractable by projection into a randomly-chosen low-dimensional subspace.
Section~\ref{sec:pnode} gives an application to the stochastic numerical solution of deterministic dynamical systems, in which the stochasticity is a device used to represent the impact of numerical discretisation uncertainty.
The proofs of all theorems are deferred to an appendix located after the bibliographic references.


\section{Setup and notation}
\label{sec:setup}

\subsection{Spaces of probability measures}

Throughout, $(\ProbSpace, \mathcal{F}, \mathbb{P})$ is a fixed probability space that is rich enough to serve as a common domain for all random variables of interest.

The space of probability measures on the Borel $\sigma$-algebra of a topological space $\UU$ will be denoted by $\ProbSimplex{\UU}$;
in practice, $\UU$ will be a separable Banach space.

When $\mu \in \ProbSimplex{\UU}$, integration of a measurable function (random variable) $f \colon \UU \to \Reals$ will also be denoted by expectation, i.e.\ $\EE_{\mu}[f] \defeq \int_{\UU} f(u) \, \rd \mu (u)$.

The space $\ProbSimplex{\UU}$ will be endowed with the Hellinger metric $\dhh \colon \ProbSimplex{\UU}^{2} \to \Reals_{\geq 0}$:
for $\mu, \nu \in \ProbSimplex{\UU}$ that are both absolutely continuous with respect to a reference measure $\pi$,
\begin{equation}
	\label{eq:Hellinger}
	\dhh(\mu, \nu)^{2}
	\defeq \frac{1}{2} \int_{\UU} \Absval{ \sqrt{ \frac{\rd \mu}{\rd \pi} } - \sqrt{ \frac{\rd \nu}{\rd \pi} } \, }^{2} \, \rd \pi
	= 1 - \int_{\UU} \sqrt{ \frac{\rd \mu}{\rd \pi} \frac{\rd \nu}{\rd \pi} } \, \rd \pi
	= 1 - \EE_{\nu} \biggl[ \sqrt{ \frac{\rd \mu}{\rd \nu} } \, \biggr] .
\end{equation}
The Hellinger distance is in fact independent of the choice of reference measure $\pi$ and defines a metric on $\ProbSimplex{\UU}$ \citep[Lemma 4.7.35--36]{Bogachev:2007} with respect to which $\ProbSimplex{\UU}$ evidently has diameter at most $1$.
The Hellinger topology coincides with the total variation topology \citep{Kraft:1955} and is strictly weaker than the Kullback--Leibler (relative entropy) topology \citep{Pinsker:1964};
all these topologies are strictly stronger than the topology of weak convergence of measures.

As used in Sections~\ref{sec:general}--\ref{sec:pnode}, the Hellinger metric is useful for uncertainty quantification when assessing the similarity of Bayesian posterior probability distributions, since expected values of square-integrable functions are Lipschitz continuous with respect to the Hellinger metric:
\begin{equation}
	\label{eq:Hellinger_bound}
	\bigabsval{ \EE_{\mu}[f] - \EE_{\nu}[f] } \leq 2 \sqrt{ \EE_{\mu} \bigl[ \absval{ f }^{2} \bigr] + \EE_{\nu} \bigl[ \absval{ f }^{2} \bigr] } \, \dhh(\mu, \nu)
\end{equation}
when $f \in L^{2}_{\mu}(\UU) \cap L^{2}_{\nu}(\UU)$.
In particular, for bounded $f$, $\absval{ \EE_{\mu}[f] - \EE_{\nu}[f] } \leq 2 \sqrt{2} \norm{ f }_{\infty} \dhh(\mu, \nu)$.

\subsection{Bayesian inverse problems}

By an \emph{inverse problem} we mean the recovery of $u \in \UU$ from an imperfect observation $y \in \YY$ of $G(u)$, for a known forward operator $G \colon \UU \to \YY$.
In practice, the operator $G$ may arise as the composition $G = O \circ S$ of the solution operator $S \colon \UU \to \mathcal{V}$ of a system of ordinary or partial differential equations with an observation operator $O \colon \mathcal{V} \to \mathcal{Y}$, and it is typically the case that $\YY = \Reals^{J}$ for some $J \in \Naturals$, whereas $\UU$ and $\VV$ can have infinite dimension.
For simplicity, we assume an additive noise model
\begin{equation}
	\label{eq:forward}
	y = G(u) + \eta ,
\end{equation}
where the statistics but not the realisation of $\eta$ are known.
In the strict sense, this inverse problem is ill-posed in the sense that there may be no element $u \in \UU$ for which $G(u) = y$, or there may be multiple such $u$ that are highly sensitive to the observed data $y$.

The Bayesian perspective eases these problems by interpreting $u$, $y$, and $\eta$ all as random variables or fields.
Through knowledge of the distribution of $\eta$, \eqref{eq:forward} defines the conditional distribution of $y|u$.
After positing a prior probability distribution $\mu_{0} \in \ProbSimplex{\UU}$ for $u$, the Bayesian solution to the inverse problem is nothing other than the posterior distribution for the conditioned random variable $u|y$.
This posterior measure, which we denote $\mu^{y} \in \ProbSimplex{\UU}$, is from the Bayesian point of view the proper synthesis of the prior information in $\mu_{0}$ with the observed data $y$.
The same posterior $\mu^{y}$ can also be arrived at via the minimisation of penalised Kullback--Leibler, $\chi^{2}$, or Dirichlet energies \citep{DupuisEllis:1997, JordanKinderlehrer:1996, OhtaTakatsu:2011}, where the penalisation again expresses compromise between fidelity to the prior and fidelity to the data.

The rigorous formulation of Bayes' formula for this context requires careful treatment and some further notation \citep{Stuart:2010}.
The pair $(u, y)$ is assumed to be a well-defined random variable with values in $\UU \times \YY$.
The marginal distribution of $u$ is the Bayesian prior $\mu_{0} \in \ProbSimplex{\UU}$.
The observational noise $\eta$ is distributed according to $\mathbb{Q}_{0} \in \ProbSimplex{\YY}$, independently of $u$.
The random variable $y|u$ is distributed according to $\mathbb{Q}_{u}$, the translate of $\mathbb{Q}_{0}$ by $G(u)$, which is assumed to be absolutely continuous with respect to $\mathbb{Q}_{0}$, with
\[
	\frac{\rd \mathbb{Q}_{u}}{\rd \mathbb{Q}_{0}} (y) \propto \exp(-\Misfit(u; y)) .
\]
The function $\Misfit \colon \UU \times \YY \to \Reals$ is called the \emph{negative log-likelihood} or simply \emph{potential}.
In the elementary setting of centred Gaussian noise, $\eta \sim \Normal(0, \Gamma)$ on $\YY = \Reals^{J}$, the potential is the non-negative quadratic misfit\footnote{Hereafter, to reduce notational clutter, we write both $\norm{ \quark }_{\UU}$ and $\norm{ \quark }_{\YY}$ as $\norm{ \quark }$.} $\Misfit(u; y) = \tfrac{1}{2} \bignorm{ \Gamma^{-1/2} ( y - G(u) ) }_{\YY}^{2}$.
However, particularly for cases in which $\dim \YY = \infty$, it may be necessary to allow $\Misfit$ to take negative values and even to be unbounded below \citep[Remark~3.8]{Stuart:2010}.

With this notation, Bayes' theorem is then as follows \citep[Theorem~3.4]{DashtiStuart:2016}:

\begin{theorem}[Generalised Bayesian formula]
	\label{thm:Bayes}
	Suppose that $\Misfit \colon \UU \times \YY \to \Reals$ is $\mu_{0}\otimes \mathbb{Q}_0$-measurable and that
	\[
		Z(y) \defeq \EE_{\mu_{0}} \bigl[ \exp(- \Misfit(u; y)) \bigr]
	\]
	satisfies $0 < Z(y) < \infty$ for $\mathbb{Q}_{0}$-almost all $y \in \YY$.
	Then, for such $y$, the conditional distribution $\mu^{y}$ of $u|y$ exists and is absolutely continuous with respect to $\mu_{0}$ with density
	\begin{equation}
		\label{eq:rad_nik}
		\frac{\rd \mu^{y}}{\rd \mu_{0}} (u) = \frac{\exp(- \Misfit(u; y))}{Z(y)} .
	\end{equation}
\end{theorem}

Note that, for \eqref{eq:rad_nik} to make sense, it is essential to check that $0 < Z(y) < \infty$.
Hereafter, to save space, we regard the data $y$ as fixed, and hence write $\Misfit(u)$ in place of $\Misfit(u; y)$, $Z$ in place of $Z(y)$, and $\mu$ in place of $\mu^{y}$.
In particular, we shall redefine the negative log-likelihood as a function $\Misfit \colon \UU \to \Reals$, instead of a function $\Misfit \colon \UU\times \YY\to\Reals$ as in Theorem~\ref{thm:Bayes} above.

From the perspective of numerical analysis, it is natural to ask about the well-posedness of the Bayesian posterior $\mu$:
is it stable when the prior $\mu_{0}$, the potential $\Misfit$, or the observed data $y$ are slightly perturbed, e.g.\ due to discretisation, truncation, or other numerical errors?
For example, what is the impact of using an approximate numerical forward operator $G_{N}$ in place of $G$, and hence an approximate $\Misfit_{N} \colon \UU \to \Reals$ in place of $\Misfit$?
Here, we quantify stability in the Hellinger metric $\dhh$ from \eqref{eq:Hellinger}.

Stability of the posterior with respect to the observed data $y$ and the log-likelihood $\Misfit$ was established for Gaussian priors by \citet{Stuart:2010} and for more general priors by many later contributions \citep{DashtiHarrisStuart:2012, Hosseini:2017, HosseiniNigam:2017, Sullivan:2017}.
(We note in passing that the stability of BIPs with respect to perturbation of the prior is possible but much harder to establish, particularly when the data $y$ are highly informative and the normalisation constant $Z(y)$ is close to zero;
see e.g.\ the ``brittleness'' phenomenon of \citep{OwhadiScovel:2017, OwhadiScovelSullivan:2015}.)
Typical approximation theorems for the replacement of the potential $\Misfit$ by a deterministic approximate potential $\Misfit_{N}$, leading to an approximate posterior $\mu_{N}$, aim to transfer the convergence rate of the forward problem to the inverse problem, i.e.\ to prove an implication of the form
\[
	\bigabsval{ \Misfit(u) - \Misfit_{N}(u) } \leq M(\norm{ u }) \psi(N) \implies \dhh \bigl( \mu, \mu_{N} \bigr) \leq C \psi(N) ,
\]
where $M \colon \Reals_{\geq 0} \to \Reals_{\geq 0}$ is suitably well behaved, $\psi \colon \Naturals \to \Reals_{\geq 0}$ quantifies the convergence rate of the forward problem, and $C$ is a constant.
Following \citet{StuartTeckentrup:2017}, the purpose of this article is to extend this paradigm and these approximation results to the case in which the approximation $\Misfit_{N}$ is a \emph{random} object.

\section{Well-posed Bayesian inverse problems with random likelihoods}
\label{sec:general}

In many practical applications, the negative log-likelihood $\Misfit$ is computationally too expensive or impossible to evaluate exactly;
one therefore often uses an approximation $\Misfit_{N}$ of $\Misfit$.
This leads to an approximation $\mu_{N}$ of the exact posterior $\mu$, and a key desideratum is convergence, in a suitable sense, of $\mu_{N}$ to $\mu$ as the approximation error $\Misfit_{N} - \Misfit$ in the potential tends to zero.

The focus of this work is on random approximations $\Misfit_{N}$.
One particular example of such random approximations are the GP emulators analysed in \citet{StuartTeckentrup:2017};
other examples include the randomised misfit models in Section~\ref{sec:random-misfit} and the probabilistic numerical methods in Section~\ref{sec:pnode}.
The present section extends the analysis of \citet{StuartTeckentrup:2017} from the case of GP approximations of forward models or log-likelihoods to more general non-Gaussian approximations.
In doing so, more precise conditions are obtained for the exact Bayesian posterior to be well approximated by its random counterpart.

Let now $\Misfit_{N} \colon \Omega \times \UU \to \R$ be a measurable function that provides a random approximation to $\Misfit \colon \UU \to \R$, where we recall that we have fixed the data $y$.
Let $\nu_{N}$ be a probability measure on $\Omega$ such that the distribution of the inputs of $\Misfit_{N}$ is given by $\nu_{N} \otimes \mu_0$;
we sometimes abuse notation and think of $\Misfit_{N}$ itself as being $\nu_{N}$-distributed.
We assume throughout that the randomness in the approximation $\Misfit_{N}$ of $\Misfit$ is independent of the randomness in the parameters being inferred.

Replacing $\Misfit$ by $\Misfit_{N}$ in \eqref{eq:rad_nik}, we obtain the \emph{sample approximation} $\mu_{N}^{\sample}$, the random measure given by
\begin{align}
	\label{eq:random_posterior}
	\frac{\rd \mu_{N}^{\sample}}{\rd \mu_{0}}(\omega, u) & \defeq \frac{\exp( -\Misfit_{N}(\omega, u) )}{Z_{N}^{\sample}} , \\
	Z_{N}^{\sample} (\omega) & \defeq \EE_{\mu_{0}} \bigl[ \exp( -\Misfit_{N}(\omega, \quark) ) \bigr]=\int_{\UU}\exp(-\Misfit_{N}(\omega, u')) \, \rd \mu_0(u') .
	\notag
\end{align}
(Henceforth, we will omit the $\omega$ argument for brevity.)
Thus, the measure $\mu$ is approximated by the random measure $\mu_{N}^{\sample} \colon \Omega \to \ProbSimplex{\UU}$, and the normalisation constant $Z_{N}^{\sample} \colon \Omega\to\Reals$ is a random variable.
A deterministic approximation of the posterior distribution $\mu$ can now be obtained either by fixing $\omega$, i.e.\ by taking one particular realisation of the random posterior $\mu_{N}^{\sample}$, or by taking the expected value of the random likelihood $\exp( -\Misfit_{N}(u) )$, i.e.\ by averaging over different realisations of $\mu_{N}^{\sample}$.
This yields the \emph{marginal approximation} $\mu_{N}^{\marginal}$ defined by
\begin{align}
	\label{eq:marginal_posterior}
	\frac{\rd \mu_{N}^{\marginal}}{\rd \mu_{0}}(u) & \defeq \frac{\EE_{\nu_{N}} \bigl[ \exp ( -\Misfit_{N} (u) ) \bigr]}{\EE_{\nu_{N}} \bigl[ Z_{N}^{\sample} \bigr]} ,
\end{align}
where $\EE_{\nu_{N}}[Z_{N}^{\sample}]=\int_{\Omega}Z_{N}^{\sample}(\omega) \, \rd \nu_{N}(\omega)$.
We note that an alternative averaged, deterministic approximation can be obtained by taking the expected value of the density $(Z_{N}^{\sample})^{-1} e^{-\Misfit_{N}(u)}$ in \eqref{eq:random_posterior} as a whole, i.e.\ by taking the expected value of the ratio rather than the ratio of expected values.
A result very similar to Theorem \ref{thm:hell_conv_marg}, with slightly modified assumptions, holds also in this case, with the proof following the same steps.
However, the marginal approximation presented here appears more intuitive and more amenable to applications.
Firstly, the marginal approximation provides a clear interpretation as the posterior distribution obtained by the approximation of the true data likelihood $\exp(-\Misfit(u))$ by $\EE_{\nu_{N}} \bigl[ \exp ( -\Misfit_{N} (u) ) \bigr]$.
Secondly, the marginal approximation is more amenable to sampling methods such as Markov chain Monte Carlo, with clear connections to the pseudo-marginal approach \citep{AndrieuRoberts:2009, Beaumont:2003}.

%

\subsection{Random misfit models}

This section considers the general setting in which the deterministic potential $\Misfit$ is approximated by a random potential $\Misfit_{N} \sim \nu_{N}$.
Recall from \eqref{eq:rad_nik} that $Z$ is the normalisation constant of $\mu$, and that for $\mu$ to be well-defined, we must have that $0<Z<\infty$.
The following two results, Theorems~\ref{thm:hell_conv_marg} and \ref{thm:hell_conv_rand}, extend Theorems~4.9 and 4.11 respectively of \citet{StuartTeckentrup:2017}, in which the approximation is a GP model:

\begin{theorem}[Deterministic convergence of the marginal posterior]
	\label{thm:hell_conv_marg}
	Suppose that there exist scalars $C_1, C_2, C_3 \geq 0$, independent of $N$, such that, for the H\"{o}lder-conjugate exponent pairs $(p_1,p_1')$, $(p_2,p_2')$, and $(p_3,p_3')$, we have
	\begin{compactenum}[(a)]
		\item \label{item:hell_conv_marg_a} $\min \left\{ \bignorm{ \EE_{\nu_{N}} [\exp(- \Misfit_{N})]^{-1} }_{L^{p_1}_{\mu_{0}}(\UU)}, \bignorm{ \exp( \Misfit) }_{L^{p_1}_{\mu_{0}}(\UU)} \right\} \leq C_1(p_1)$;
		\item \label{item:hell_conv_marg_b} $\Norm{ \EE_{\nu_{N}} \Big[ \big(\exp( -\Misfit ) + \exp ( -\Misfit_{N} ) \big)^{p_2} \Big]^{1/p_2} }_{L^{2p_1'p_3}_{\mu_{0}}(\UU)} \leq C_2(p_1,p_2,p_3)$;
		\item \label{item:hell_conv_marg_c} $C_3^{-1} \leq \EE_{\nu_{N}} [Z_{N}^{\sample}] \leq C_3$.
	\end{compactenum}
	Then there exists $C=C(C_1,C_2,C_3,Z)>0$, independent of $N$, such that 
	\begin{subequations}
		\label{eq:hell_conv_marg_result}
		\begin{align}
			\label{eq:hell_conv_marg_result_inequality}
			\dhh \bigl( \mu, \mu_{N}^{\marginal} \bigr) &\leq C \Norm{ \EE_{\nu_{N}} \bigl[ \absval{ \Misfit - \Misfit_{N} }^{p_2'} \bigr]^{1/p_2'} }_{L^{2p_1' p_3'}_{\mu_{0}}(\UU)},
			\\
			\label{eq:hell_conv_marg_result_constant_C}
		 	C(C_1,C_2,C_3,Z)&=\left(\frac{C_1(p_1)}{Z}+ C_3\max \left\{ Z^{-3}, C_3^3 \right\} \right) C_2^2(p_1,p_2,p_3).
		\end{align}
	\end{subequations}
\end{theorem}

In the proof of Theorem~\ref{thm:hell_conv_marg}, we show that hypothesis (\ref{item:hell_conv_marg_a}) arises as an upper bound on the quantity $\norm{(e^{-\Misfit}+\EE_{\nu_{N}}[e^{-\Misfit_{N}}])^{-1}}_{L^{p_1}_{\mu_0}(\UU)}$.
In order for the conclusion of Theorem~\ref{thm:hell_conv_marg} to hold, we need the latter to be finite.
Thus, hypothesis (\ref{item:hell_conv_marg_a}) is an exponential decay condition on the \emph{positive} tails of either $\Misfit$ or $\Misfit_{N}$, with respect to the appropriate measures.
Alternatively, by applying Jensen's inequality to $\EE_{\nu_{N}}[e^{-\Misfit_{N}}]^{-1}$, one can strengthen hypothesis (\ref{item:hell_conv_marg_a}) into the hypothesis of exponential integrability of either $\Misfit$ with respect to $\mu_0$ or $\Misfit_{N}$ with respect to $\nu_{N} \otimes \mu_{0}$; this yields the same interpretation.
Thus, the parameter $p_1$ quantifies the exponential decay of the positive tail of either $\Misfit$ or $\Misfit_{N}$.

By comparing the quantity $\norm{(e^{-\Misfit}+\EE_{\nu_{N}}[e^{-\Misfit_{N}}])^{-1}}_{L^{p_1}_{\mu_0}(\UU)}$ from hypothesis (\ref{item:hell_conv_marg_a}) with the quantity in hypothesis (\ref{item:hell_conv_marg_b}), it follows that hypothesis (\ref{item:hell_conv_marg_b}) is an exponential decay condition on the \emph{negative} tails of both $\Misfit$ and $\Misfit_{N}$.
The two new parameters in this decay condition arise because we apply H\"{o}lder's inequality twice in order to develop the desired bound \eqref{eq:hell_conv_marg_result} on $\dhh(\mu,\mu_{N}^{\marginal})$.
The key desideratum here is that the bound is \emph{multiplicative} in some $L^{p'}_{\mu_0}(\UU)$-norm of $\EE_{\nu_N}[ \absval{ \Misfit-\Misfit_{N} }^{p_2'}]^{1/p_2'}$.
The two new parameters $p_2$ and $p_1'p_3$ quantify the decay with respect to $\nu_{N}$ and $\mu_0$ respectively.
Note that the interaction between the hypotheses (\ref{item:hell_conv_marg_a}) and (\ref{item:hell_conv_marg_b}) as described by the conjugate exponent pair $(p_1,p_1')$ implies that one can trade off faster exponential decay of one tail with slower exponential decay of the other.

The two-sided condition on $\EE_{\nu_N}[Z_{N}^{\sample}]$ in hypothesis (\ref{item:hell_conv_marg_c}) ensures that both tails of $\Misfit_{N}$ with respect to $\nu_{N} \otimes \mu_0$ decay sufficiently quickly.
This hypothesis ensures that the Radon--Nikodym derivative in \eqref{eq:marginal_posterior} is well-defined.

Finally, we note that the quantity on the right hand side of \eqref{eq:hell_conv_marg_result_inequality} depends directly on the conjugate exponents of $p_1, p_2$ and $p_3$ appearing in hypotheses (\ref{item:hell_conv_marg_a}) and (\ref{item:hell_conv_marg_b}).
The more well behaved the quantities in these hypotheses are, the weaker the norm we can choose on the right hand side of \eqref{eq:hell_conv_marg_result_inequality}.

\begin{theorem}[Mean-square convergence of the sample posterior]
	\label{thm:hell_conv_rand}
	Suppose that there exist scalars $D_1, D_2 \geq 0$, independent of $N$, such that, for H\"{o}lder-conjugate exponent pairs $(q_1,q_1')$ and $(q_2,q_2')$, we have
	\begin{compactenum}[(a)]
		\item \label{item:hell_conv_rand_a} $\Norm{ \EE_{\nu_{N}} \Big[ \big( e^{-\Misfit/2} + e^{- \Misfit_{N}/2} \big)^{2q_1}\Big]^{1/q_1} }_{L^{q_2}_{\mu_{0}}(\UU)} \leq D_1(q_1,q_2)$;
		\item \label{item:hell_conv_rand_b} $\Norm{ \EE_{\nu_{N}} \left[ \left(Z_{N}^{\sample} \max\left\{Z^{-3},(Z_{N}^{\sample})^{-3}\right\} \left( e^{- \Misfit} + e^{-\Misfit_{N}} \right)^{2} \right)^{q_1}\right]^{1/q_1} }_{L^{q_2}_{\mu_{0}}(\UU)} \leq D_2(q_1,q_2)$.
	\end{compactenum}
	Then 
	\begin{align}
	 	\label{eq:hell_conv_rand_result}
		\EE_{\nu_{N}} \left[ \dhh \bigl( \mu, \mu_{N}^{\sample} \bigr)^2 \right]^{1/2} &\leq \left(D_1+D_2\right) \Norm{ \EE_{\nu_{N}} \left[ \absval{ \Misfit - \Misfit_{N} }^{2q_1'} \right]^{1/2q_1'} }_{L^{2q_2'}_{\mu_{0}}(\UU)},
	\end{align}
\end{theorem}

Hypothesis (\ref{item:hell_conv_rand_a}) of Theorem~\ref{thm:hell_conv_rand} arises during the proof as a result of developing an upper bound on $\norm{\EE_{\nu_N}[(e^{-\Misfit/2}-e^{-\Misfit_{N}/2})^2]}$ that is multiplicative in some $L^{p'}_{\mu_0}(\UU)$-norm of $\norm{\EE_{\nu_N}[\absval{ \Misfit-\Misfit_N }^{2q_1'}]^{1/2q_1'}}$.
Thus, it describes an exponential decay condition of the negative tails of both $\Misfit$ or $\Misfit_N$;
in particular, hypothesis (\ref{item:hell_conv_rand_a}) is always satisfied when the potentials $\Misfit$ or $\Misfit_N$ are non-negative, as is usually the case for finite-dimensional data.
The appearance of $q_1$ and $q_2$ arises due to one application of H\"{o}lder's inequality for fulfulling the desideratum of multiplicativity, and $q_1$ and $q_2$ quantify the decay with respect to $\nu_N$ and $\mu_0$ respectively.

Hypothesis (\ref{item:hell_conv_rand_b}) of Theorem~\ref{thm:hell_conv_rand} arises as a result of developing an upper bound on the quantity $\EE_{\nu_N}[Z_{N}^{\sample}(Z^{-1/2}-(Z_{N}^{\sample})^{-1/2})^2]$ that fulfills the desideratum of multiplicativity mentioned above.
The presence of both $Z_{N}^{\sample}$ and its reciprocal indicates that hypothesis (\ref{item:hell_conv_rand_b}) is analogous to hypothesis (\ref{item:hell_conv_marg_c}) of Theorem~\ref{thm:hell_conv_marg}, in that hypothesis (\ref{item:hell_conv_rand_b}) is a condition on the tails of $\Misfit_{N}$ with respect to $\mu_0$.
The difference between hypothesis (\ref{item:hell_conv_rand_b}) of Theorem~\ref{thm:hell_conv_rand} and hypothesis (\ref{item:hell_conv_marg_c}) of Theorem~\ref{thm:hell_conv_marg} arises due to the fact that the Radon--Nikodym derivative in \eqref{eq:random_posterior} features $Z_{N}^{\sample}$ instead of $\EE_{\nu_N}[Z_{N}^{\sample}]$.

We now show that the assumptions of Theorems~\ref{thm:hell_conv_marg} and \ref{thm:hell_conv_rand} are satisfied when the exact potential $\Misfit$ and the approximation quality $\Misfit_{N} \approx \Misfit$ are suitably well behaved.
Since $0 < Z < \infty$, it follows that $C_3^{-1} < Z < C_3$ for some $0 < C_{3} < \infty$.

\begin{assumption}
	\label{asmp:goodsuffcon_for_hellconvthms}
	There exists $C_0\in\Reals$ that does not depend on $N$, such that, for all $N \in \Naturals$,
	\begin{equation}
		\label{eq:lower_bounds_on_potentials_Phi_PhiN}
		\Misfit\geq -C_0\quad\text{and}\quad \nu_{N}\left(\{\Misfit_{N} \mid \Misfit_{N}\geq -C_0\}\right)=1,
	\end{equation}
	and for any $0 < C_3 < \infty$ with the property that $C_3^{-1}<Z<C_3$, there exists $N^\ast(C_3)\in\Naturals$ such that, for all $N \geq N^{\ast}$,
	\begin{equation}
		\label{eq:boundedness_of_EmuEnu_N_Phi_N_errorsequence}
		\EE_{\mu_{0}}\left[\EE_{\nu_{N}}\left[ \absval{ \Misfit_{N} - \Misfit} \right]\right] \leq \frac{1}{2\exp(C_0)} \min \left\{Z-\frac{1}{C_3},C_3-Z\right\} .
	\end{equation}
\end{assumption}

The lower bound conditions in \eqref{eq:lower_bounds_on_potentials_Phi_PhiN} ensure that the hypothesised exponential decay conditions on the negative tails of the true likelihood and the random likelihoods from Theorems~\ref{thm:hell_conv_marg} and \ref{thm:hell_conv_rand} are satisfied.
The uniform lower bound on $\Misfit$ translates into a uniform upper bound of the Radon--Nikodym derivative of the posterior with respect to the prior, and is a very mild condition that is satisfied in many, if not most, BIPs.
Given this fact, it is reasonable to demand that the $\Misfit_{N}$ satisfy the same uniform lower bound, $\nu_{N}$-almost surely and for all $N\in\mathbb{N}$; this is the content of the second condition in \eqref{eq:lower_bounds_on_potentials_Phi_PhiN}.
Condition \eqref{eq:boundedness_of_EmuEnu_N_Phi_N_errorsequence} expresses the condition that, by choosing $N$ sufficiently large, one can approximate $\Misfit$ arbitrarily well using the random $\Misfit_{N}$, with respect to the $L^1_{\mu_{0} \otimes \nu_{N}}$ topology.
This assumption ensures that the stated aims of this work are reasonable.

\begin{lemma}
	\label{lemma:goodsuffcon_for_hell_conv_thms_for_asmp_set_1}
	Suppose that Assumption~\ref{asmp:goodsuffcon_for_hellconvthms} holds with $C_0$ as in \eqref{eq:lower_bounds_on_potentials_Phi_PhiN} and $C_3$ and $N^\ast(C_3)$ as in \eqref{eq:boundedness_of_EmuEnu_N_Phi_N_errorsequence}, that $\exp(\Misfit)\in L^{p^\ast}_{\mu_{0}}(\UU)$ for some $1\leq p^\ast\leq +\infty$ with conjugate exponent $(p^\ast)'$, and there exists some $C_4\in\mathbb{R}$ that does not depend on $N$, such that, for all $N \in \Naturals$,
	\begin{equation}
		\label{eq:means_of_PhiN_uniformly_bounded}
		\nu_{N}\left(\left\{ \Misfit_{N} \mid \EE_{\mu_{0}}\left[\Misfit_{N}\right]\leq C_4\right\}\right)=1 .
	\end{equation}
	Then the hypotheses of Theorem~\ref{thm:hell_conv_marg} hold, with
	\begin{equation*}
		p_1= p^\ast,\ p_2=p_3= +\infty,\ C_1 = \norm{ \exp(\Misfit) }_{L^{p^\ast}_{\mu_{0}}(\UU)},\ C_2=2\exp(C_0),
	\end{equation*}
	and $C_3$ as above.
	Moreover, the hypotheses of Theorem~\ref{thm:hell_conv_rand} hold, with 
	\begin{equation*}
		q_1=q_2=\infty,\ D_1=4\exp(C_0),\ D_2=4\exp(3C_0)\max\{C_3^{-3},\exp(3C_4)\}.
	\end{equation*}
\end{lemma}

The uniform upper bound condition on $\Misfit_{N}$ with respect to $\mu_{0}$ in \eqref{eq:means_of_PhiN_uniformly_bounded} is rather strong; we use it to ensure that $Z_{N}^{\sample}$ is bounded away from zero, uniformly with respect to $\Misfit_{N}$ and $N\in\mathbb{N}$.
Together with the condition on $\Misfit_{N}$ in \eqref{eq:lower_bounds_on_potentials_Phi_PhiN}, this translates to uniform lower and upper bounds on $Z_{N}^{\sample}$; the latter implies that hypothesis (\ref{item:hell_conv_rand_b}) in Theorem~\ref{thm:hell_conv_rand} holds with the stated values of $q_1$ and $q_2$.
A sufficient condition for \eqref{eq:means_of_PhiN_uniformly_bounded} is that the $\Misfit_{N}$ are themselves uniformly bounded.
This condition is of interest when the misfit $\Misfit$ is associated to a bounded forward model and the data take values in a bounded subset.

\begin{lemma}\label{lemma:goodsuffcon_for_hell_conv_thms_for_asmp_set_2}
	Suppose that Assumption~\ref{asmp:goodsuffcon_for_hellconvthms} holds with $C_0$ as in \eqref{eq:lower_bounds_on_potentials_Phi_PhiN} and $C_3$ and $N^\ast(C_3)$ as in \eqref{eq:boundedness_of_EmuEnu_N_Phi_N_errorsequence}, and that there exists some $2<\rho^\ast< +\infty$ such that $\EE_{\nu_{N}}[\exp(\rho^\ast\Misfit_{N})]\in L^1_{\mu_{0}}(\UU)$.
	Then the hypotheses of Theorem~\ref{thm:hell_conv_marg} hold, with 
	\begin{equation*}
		p_1=\rho^\ast,\ p_2=p_3= +\infty,\ C_1=\norm{\EE_{\nu_{N}}[\exp(\rho^\ast\Misfit_{N})]}_{L^1_{\mu_{0}}(\UU)}^{1/\rho^\ast},\ C_2=2\exp(C_0),
	\end{equation*}
	and $C_3$ as above.
	Moreover, the hypotheses of Theorem~\ref{thm:hell_conv_rand} hold, with 
	\begin{align*}
		q_1 & = \frac{\rho^\ast}{2}, &
		q_2 & = +\infty, \\
		D_1 & = 4\exp(C_0), &
		D_2 & = 4\exp(2C_0)\left(C_3^{-3}\exp(C_0)+\norm{\EE_{\nu_{N}}[\exp(\rho^\ast\Misfit_{N})]}^{2/\rho^\ast}_{L^1_{\mu_{0}}(\UU)}\right).
	\end{align*}
\end{lemma}
By comparing the hypotheses and conclusions of Lemma~\ref{lemma:goodsuffcon_for_hell_conv_thms_for_asmp_set_1} and Lemma~\ref{lemma:goodsuffcon_for_hell_conv_thms_for_asmp_set_2}, we observe that, by reducing the exponent of integrability from $q_1=+\infty$ to $q_1=\rho^\ast/2$, we can replace the strong uniform upper bound condition \eqref{eq:means_of_PhiN_uniformly_bounded} on $\Misfit_{N}$ from Lemma~\ref{lemma:goodsuffcon_for_hell_conv_thms_for_asmp_set_1} with the weaker condition that $\exp(\Misfit_{N})\in L^{\rho^\ast}_{\mu_0}(\UU)$ in Lemma~\ref{lemma:goodsuffcon_for_hell_conv_thms_for_asmp_set_2}, and thus increase the scope of applicability of the conclusion.

In Lemmas~\ref{lemma:goodsuffcon_for_hell_conv_thms_for_asmp_set_1} and \ref{lemma:goodsuffcon_for_hell_conv_thms_for_asmp_set_2} above, we have specified the largest possible values of the exponents that are compatible with the hypotheses.
This is because later, in Theorem~\ref{thm:cvgce_of_random_measures}, we will want to use the smallest possible values of the corresponding \emph{conjugate} exponents in the resulting inequalities \eqref{eq:hell_conv_marg_result_inequality} and \eqref{eq:hell_conv_rand_result}.

%

\subsection{Random forward models in quadratic potentials}

In many settings, the potentials $\Misfit$ and $\Misfit_{N}$ have a common form and differ only in the parameter-to-observable map.
In this section we shall assume that $\Misfit$ and $\Misfit_{N}$ are quadratic misfits of the form
\begin{equation}
	\label{eq:quadratic_misfits}
	\Misfit(u) = \frac{1}{2} \bignorm{ \Gamma^{-1/2} ( G(u) - y ) }^{2}
	\quad
	\text{and}
	\quad
	\Misfit_{N}(u) = \frac{1}{2} \bignorm{ \Gamma^{-1/2} ( G_{N}(u) - y ) }^{2}, 
\end{equation}
corresponding to centred Gaussian observational noise with symmetric positive-definite covariance $\Gamma$.
Again, we assume that $G$ is deterministic while $G_{N}$ is random.
In this section, for this setting, we show how the quality of the approximation $G_{N} \approx G$ transfers to the approximation $\Misfit_{N} \approx \Misfit$, and hence to the approximation $\mu_{N} \approx \mu$ (for either the sample or marginal approximate posterior).

Pointwise in $u$ and $\omega$, the errors in the misfit and the forward model are related according to the following proposition.

\begin{proposition}
	\label{proposition:bound_on_misfiterror_intermsof_forwardmodelerror}
	Let $\Misfit$ and $\Misfit_{N}$ be defined as in \eqref{eq:quadratic_misfits}, where $\YY=\Reals^J$ for some $J\in\Naturals$ and the eigenvalues of the operator $\Gamma$ are bounded away from zero.
	Then, for some $C = C_{\Gamma} >0$, for all $u \in \mathcal{U}$, and $\nu_N$-almost surely
	\begin{equation}
		\label{eq:ae_in_u_bound_absval_error_potentials_01}
		\bigabsval{ \Misfit(u) - \Misfit_{N}(u) }\leq 2 C_{\Gamma} \left(\Misfit(u)^{1/2}\norm{G(u)-G_N(u)}+\norm{G(u)-G_{N}(u)}^2\right) .
	\end{equation}
	Hence, for $q\in[1,\infty)$ and all $u \in \UU$,
	\begin{align}
		\label{eq:ae_in_u_bound_absval_error_potentials_01b}
		\EE_{\nu_{N}}\left[\bigabsval{\Misfit(u)-\Misfit_{N}(u)}^q\right]^{1/q}
		& \leq 4C_\Gamma \Bigl( \Misfit(u)^{q/2}\EE_{\nu_{N}}\left[\norm{G(u)-G_N(u)}^q\right] \\
		& \phantom{=} \quad + \EE_{\nu_{N}}\left[\norm{G(u)-G_{N}(u)}^{2q}\right] \Bigr)^{1/q}.
		\notag
	\end{align}
\end{proposition}
By assuming that $\YY=\Reals^J$, we assume that the data live in a finite-dimensional space.
This is a standard assumption in the area, and implies that the operator $\Gamma$ is simply a matrix.
The assumption of the eigenvalues of $\Gamma$ being bounded away from zero is equivalent to assuming that $\Gamma$ is invertible, which follows immediately from the assumption stated earlier that $\Gamma$ is a symmetric and positive-definite covariance matrix.

\begin{corollary}
	\label{corollary:bound_on_misfiterror_intermsof_forwardmodelerror}
	Let $1\leq q\leq s$, and suppose that $\Misfit\in L^s_{\mu_{0}}(\UU)$.
	If there exists an $N^\ast\in\Naturals$ such that, for all $N \geq N^{\ast}$,
	\begin{equation*}
		\label{eq:hypothesis_Emu0_EnuN_error_forwardmodel_leq1}
		\Norm{ \EE_{\nu_{N}} \bigl[ \norm{ G - G_N }^{2 q} \bigr]^{1/q} }_{L_{\mu_{0}}^{s}(\UU)} \leq 1,
	\end{equation*}
	then, there exists some $C=C(s)>0$ that does not depend on $N$ such that for all $N \geq N^\ast$, 
	\begin{align*}
		\Norm{ \EE_{\nu_{N}} \bigl[ \absval{ \Misfit-\Misfit_{N} }^q \bigr]^{1/q} }_{L^s_{\mu_{0}}(\UU)} & \leq C\Norm{ \EE_{\nu_{N}} \bigl[ \norm{ G - G_N }^{2q} \bigr]^{1/q} }_{L^s_{\mu_{0}}(\UU)}^{1/2}		
	\end{align*}
	where $C(s)=(8C_\Gamma)\left(\EE_{\mu_{0}} [ \Misfit^{s} ]^{1/2}+1\right)^{1/s}$ and $C_\Gamma$ is as in Proposition~\ref{proposition:bound_on_misfiterror_intermsof_forwardmodelerror}.
\end{corollary}
The hypotheses ensure that the integrability of the misfit $\Misfit$ determines the highest degree of integrability of the forward operators $G_{N}$ and $G$, and that for sufficiently large $N$, we may make the norm of the difference of $G-G_N$ in an appropriate topology small enough.
The constraint \eqref{eq:hypothesis_Emu0_EnuN_error_forwardmodel_leq1} is used to combine the $\norm{ G(u)-G_N(u) }$ and $\norm{ G(u)-G_{N}(u) }^2$ terms in \eqref{eq:ae_in_u_bound_absval_error_potentials_01}.
The resulting simplification ensures that we may apply Lemma~\ref{lemma:necessary_condition_of_convergence_of_forward_models}.

\begin{lemma}
	\label{lemma:necessary_condition_of_convergence_of_forward_models}
	Let $\Misfit$ and $\Misfit_{N}$ be as in \eqref{eq:quadratic_misfits}.
	If, for some $q, s \geq 1$,
	\begin{equation}
			\label{eq:convergence_hypothesis_of_forward_models}
			\lim_{N\to\infty} \Norm{ \EE_{\nu_{N}} \bigl[ \norm{ G - G_N }^{2q} \bigr]^{1/q} }_{L^{s}_{\mu_{0}}(\UU)} = 0,
	\end{equation} 
	then Assumption~\ref{asmp:goodsuffcon_for_hellconvthms} holds.
\end{lemma}
The lemma states that if the random forward model converges to the true forward model in the appropriate topology, then the conditions in Assumption \ref{asmp:goodsuffcon_for_hellconvthms} are satisfied by the corresponding random misfits.
Since the misfits were assumed to be quadratic in \eqref{eq:quadratic_misfits}, the key contribution of Lemma~\ref{lemma:necessary_condition_of_convergence_of_forward_models} is to ensure that the approximation quality condition \eqref{eq:boundedness_of_EmuEnu_N_Phi_N_errorsequence} is satisfied.

We shall use the preceding results to obtain bounds on the Hellinger distance in terms of errors in the forward model, of the following form: for $C,D>0$ and $r_1,r_2,s_1,s_2\geq 1$ that do not depend on $N$,
\begin{align}
	\dhh \bigl( \mu, \mu_{N}^{\marginal} \bigr)&\leq C \bignorm{ \EE_{\nu_{N}}\left[\norm{ G_{N} - G }^{2r_1}\right]^{1/r_1} }^{1/2}_{L^{r_2}_{\mu_{0}}(\UU)}
	\label{eq:hell_conv_marg_result_forward_model}
	\\
	\EE_{\nu_{N}} \left[ \dhh \bigl( \mu, \mu_{N}^{\sample} \bigr)^2 \right]^{1/2} &\leq D \bignorm{ \EE_{\nu_{N}}\left[\norm{ G_{N} - G }^{2s_1}\right]^{1/s_1} }^{1/2}_{L^{s_2}_{\mu_{0}}(\UU)}.
	\label{eq:hell_conv_rand_result_forward_model}
\end{align}
For brevity and simplicity, the following result uses one pair $q,s\geq 1$ in \eqref{eq:convergence_hypothesis_of_forward_models} in order to obtain convergence statements for \emph{both} $\mu_{N}^{\marginal}$ and $\mu^{\sample}_{N}$.
If one is interested in only one of these measures, then one may optimise $q$ and $s$ accordingly.

\begin{theorem}[Convergence of posteriors for randomised forward models in quadratic potentials]
	\label{thm:cvgce_of_random_measures}
	Let $\Misfit$ and $\Misfit_{N}$ be as in \eqref{eq:quadratic_misfits}.
	\begin{compactenum}[(a)]
		\item \label{thm:cvgce_of_random_measures_1} Suppose there exists some $p^\ast> 1$ with H\"{o}lder conjugate $(p^\ast)'$ such that $\exp(\Misfit)\in L^{p^\ast}_{\mu_{0}}(\UU)$, and suppose that \eqref{eq:means_of_PhiN_uniformly_bounded} holds for some $C_4>0$.
		If $G_{N} \to G$ as in \eqref{eq:convergence_hypothesis_of_forward_models} with $q=2$ and $s=2p^\ast/(p^\ast-1)$, then the following hold:
		\begin{compactenum}[(i)]
		\item there exists some $C>0$ that does not depend on $N$, for which \eqref{eq:hell_conv_marg_result_forward_model} holds with $r_1=1$ and $r_2=2p^\ast/(p^\ast-1)$, and
		\item there exists some $D>0$ that does not depend on $N$, for which \eqref{eq:hell_conv_rand_result_forward_model} holds with $s_1=2$ and $s_2=2$.
		\end{compactenum}
		\item \label{thm:cvgce_of_random_measures_2} Suppose there exists some $2<\rho^\ast<\infty$ such that $\EE_{\nu_{N}}[\exp(\rho^\ast\Misfit_{N})]\in L^1_{\mu_{0}}$.
		If $G_{N} \to G$ as in \eqref{eq:convergence_hypothesis_of_forward_models} with $q=2\rho^\ast/(\rho^\ast-2)$ and $s=2\rho^\ast/(\rho^\ast-1)$, then the following hold:
		\begin{compactenum}[(i)]
		\item there exists some $C>0$ that does not depend on $N$, for which \eqref{eq:hell_conv_marg_result_forward_model} holds with $r_1=1$ and $r_2=2\rho^\ast/(\rho^\ast-1)$, and 
		\item there exists some $D>0$ that does not depend on $N$, for which \eqref{eq:hell_conv_rand_result_forward_model} holds with $s_1=2\rho^\ast/(\rho^\ast-2)$ and $s_2=2$.
		\end{compactenum}
	\end{compactenum}
	In both cases, $\mu^{\marginal}_N$ and $\mu^{\sample}_N$ converge to $\mu$ in the appropriate metrics given in \eqref{eq:hell_conv_marg_result_forward_model} and \eqref{eq:hell_conv_rand_result_forward_model} respectively.
\end{theorem}

The proof of Theorem \ref{thm:cvgce_of_random_measures} consists of tracking the dependence of the parameters over the sequential application of the preceding results, all of which are used.

Case (\ref{thm:cvgce_of_random_measures_1}) applies in the situation where the random approximations $\Misfit_{N}$ are uniformly bounded from above; as discussed earlier, this condition is satisfied in the case that the misfit $\Misfit$ is associated to a bounded forward model and the data take values in a bounded subset of $\YY=\Reals^{J}$.
Note that the topology of the convergence of $G_{N}$ to $G$ is quantified by $s$ and $q$, and that $s$ depends on the parameter $p^{\ast}$ that quantifies the exponential $\mu_{0}$-integrability of the misfit $\Misfit$.
In particular, the faster the exponential decay of the positive tail of $\Misfit$ (i.e.\ the larger the value of $p^{\ast}$), the stronger the topology of convergence of $G_{N}$ to $G$.

In contrast to case (\ref{thm:cvgce_of_random_measures_1}), case (\ref{thm:cvgce_of_random_measures_2}) does not assume that the misfit $\Misfit$ is exponentially integrable or that the random approximations $\Misfit_{N}$ are uniformly bounded from above $\nu_{N}$-almost surely.
Instead, exponential integrability of the random misfit $\Misfit_{N}$ is required.
Another difference is that the exponential integrability parameter $\rho^{\ast}$ determines the strength of the topology of convergence of the random forward models, not only with respect to the $\mu_{0}$-topology, but also to the $\nu_{N}$-topology as well.
\section{Application: randomised misfit models}
\label{sec:random-misfit}

This section considers a particular Monte Carlo approximation $\Misfit_{N}$ of a quadratic potential $\Misfit$, proposed by \citet{NemirovskiJuditskyLanShapiro:2008, ShapiroDentchevaRuszczynski:2009}, and further applied and analysed in the context of BIPs by \citet{LeMyersBuiThanhNguyen:2017}.
This approximation is particularly useful when the data $y \in \R^{J}$ has very high dimension, so that one does not wish to interrogate every component of the data vector $y$, or evaluate every component of the model prediction $G(u)$ and compare it with the corresponding component of $y$.

Let $\sigma$ be an $\Reals^{J}$-valued random vector with mean zero and identity covariance, and let $\sigma^{(1)}, \dots, \sigma^{(N)}$ be independent and identically distributed copies (samples) of $\sigma$. 
We then have the following approximation:
\begin{align*}
	\Misfit(u)
	& \defeq \frac{1}{2} \Norm{ \Gamma^{-1/2} ( y - G(u) ) }^{2} \\
	& = \frac{1}{2} \bigl( \Gamma^{-1/2} ( y - G(u) ) \bigr)^{\transpose} \EE [ \sigma \sigma^{\transpose} ] \bigl( \Gamma^{-1/2} ( y - G(u) ) \bigr) \\
	& = \frac{1}{2} \EE \biggl[ \bigabsval{ \sigma^{\transpose} \bigl( \Gamma^{-1/2} ( y - G(u) ) \bigr) }^{2} \biggr] \\
	& \approx \frac{1}{2 N} \sum_{i = 1}^{N} \bigabsval{ {\sigma^{(i)}}^{\transpose} \bigl( \Gamma^{-1/2} ( y - G(u) ) \bigr) }^{2} \\
	& \qefed \Misfit_{N}(u) .
\end{align*}
The analysis and numerical studies in \citet[Sections 3--4]{LeMyersBuiThanhNguyen:2017} suggest that a good choice for the random vector $\sigma$ would be one with independent and identically distributed (i.i.d.) entries from a sub-Gaussian probability distribution on $\Reals$.
Examples of sub-Gaussian distributions considered include
\begin{compactenum}[(a)]
	\item the standard Gaussian distribution:
	$\sigma_j \sim \Normal(0, 1)$, for $j=1, \dots, J$;
	and
	\item the $\ell$-sparse distribution:
	for $\ell \in [0,1)$, let $s \defeq \frac{1}{1-\ell} \geq 1$ and set, for $j=1, \dots, J$,
	\[
		\sigma_j \defeq \sqrt{s}
		\begin{cases}
			1, &\text{with probability $\frac{1}{2s}$,} \\
			0, &\text{with probability $\ell = 1- \frac{1}{s}$,} \\
			-1, &\text{with probability $\frac{1}{2s}$.}
		\end{cases}
	\]
\end{compactenum}

The randomised misfit $\Misfit_N$ can provide computational benefits in two ways.
Firstly, a single evaluation of $\Misfit_N$ can be made cheap by choosing the $\ell$-sparse distribution for $\sigma$, with large sparsity parameter $\ell$. 
This choice ensures that a large proportion of the entries of each sample $\sigma^{(i)}$ will be zero, significantly reducing the cost to compute the required inner products in $\Misfit_N$, since there is no need to compute the components of the data or model vector that will be eliminated by the sparsity pattern.
The value of $N$ of course also influences the computational cost.
It is observed by \citet{LeMyersBuiThanhNguyen:2017} that, for large $J$ and moderate $N \approx 10$, the random potential $\Misfit_{N}$ and the original potential $\Misfit$ are already very similar, in particular having approximately the same minimisers and minimum values.
Statistically, these correspond to the maximum likelihood estimators under $\Misfit$ and $\Misfit_{N}$ being very similar; after weighting by a prior, this corresponds to similarity of maximum a posteriori (MAP) estimators.

The second benefit of the randomised misfit approach, and the main motivation for its use in  \citet{LeMyersBuiThanhNguyen:2017}, is the reduction in computational effort needed to compute the MAP estimate.
This task involves the solution of a large-scale optimisation problem involving $\Misfit$ in the objective function, which is typically done using inexact Newton methods.
It is shown by \citet{LeMyersBuiThanhNguyen:2017} that the required number of evaluations of the forward model $G$ and its adjoint is drastically reduced when using the randomised misfit $\Misfit_N$ as opposed to using the true misfit $\Misfit$, approximately by a factor of $\frac{J}{N}$.

The aim of this section is to show that the use of the randomised misfit $\Misfit_N$ does not only lead to the MAP estimate being well-approximated, but in fact the whole Bayesian posterior distribution.
Thus, the corresponding conjecture is that the ideal and deterministic posterior $\rd \mu(u) \propto \exp( - \Misfit(u)) \, \rd \mu_{0}(u)$ is well approximated by the random posterior $\rd \mu_{N}^{\sample}(u) \propto \exp( - \Misfit_{N}(u)) \, \rd \mu_{0}(u)$.
Indeed, via Theorem~\ref{thm:hell_conv_rand}, we have the following convergence result for the case of a sparsifying distribution:

\begin{proposition}
	\label{prop:ell_sparse_conv}
	Suppose that the entries of $\sigma$ are i.i.d.\ $\ell$-sparse, for some $\ell \in [0, 1)$, and that $\Misfit \in L^{2}_{\mu_{0}}(\UU)$.
	Then there exists a constant $C$, independent of $N$, such that
	\begin{equation}
		\label{eq:ell_sparse_conv_1}
		\left(\EE_{\nu_N} \bigl[ \dhh \bigl( \mu, \mu_{N}^{\sample} \bigr)^2 \bigr] \right)^{1/2} \leq \frac{C}{\sqrt{N}}.
	\end{equation}
\end{proposition}

(In this section, $\nu_N$ plays the role of the distribution of $\sigma^{(1)}, \dots, \sigma^{(N)}$.)
As the proof reveals, a valid choice of the constant $C$ in \eqref{eq:ell_sparse_conv_1} is
\begin{equation}
	\label{eq:ell_sparse_conv_2}
	C 
	= (D_1 + D_2) \sqrt{ J^3 \EE_{\nu_N} [\sigma_j^4] - 1 } \norm{ \Misfit }_{L^{2}_{\mu_{0}}(\UU)} \\
	= (D_1 + D_2) \sqrt{ J^3 s^{3} - 1 } \norm{ \Misfit }_{L^{2}_{\mu_{0}}(\UU)},
\end{equation}
where the constant $(D_1 + D_2)$ is as in Theorem~\ref{thm:hell_conv_rand}.
Thus, as one would expect, the accuracy of the approximation decreases as $\sigma$ approaches the complete sparsification case $\ell = 1$ or as the data dimension $J$ increases, but always with the same convergence rate $N^{-1/2}$ in terms of the approximation dimension $N$.

\begin{remark}
	\label{rmk:ell_sparse_conv_Gaussian}
	The proof of Proposition~\ref{prop:ell_sparse_conv} can be modified to yield the same result for arbitrary i.i.d.\ $\sigma_j$ with bounded support, though the sparsifying case is obviously the one with the easiest interpretation.
	However, extending Proposition~\ref{prop:ell_sparse_conv} to the case of i.i.d.\ Gaussian random variables $\sigma_j \sim \mathcal N(0,1)$ appears to be problematic.
	In the proof, we crucially make use of the bound $\absval{ \sigma_j } \leq \sqrt{s}$ to verify Assumption~(\ref{item:hell_conv_rand_b}) of Theorem~\ref{thm:hell_conv_rand}.
	For Gaussian random variables, we would similarly need an \emph{$N$-independent} bound on the exponential moments of
	\[
		\max_{\substack{1 \leq i \leq N \\ 1 \leq j \leq J}} \sigma_j^{(i)} ,
	\]
	which is not possible.
	We leave this as an interesting question for future work:
	would a different proof strategy yield convergence in the Gaussian case, or is the Gaussian setting genuinely one in which the MAP problem is well approximated but the BIP is not?
\end{remark}

\section{Application: probabilistic integration of dynamical systems}
\label{sec:pnode}

The data-based inference of initial conditions or governing parameters for dynamical problems arises frequently in scientific applications, a prime example being data assimilation in numerical weather prediction \citep{LawStuartZygalakis:2015, ReichCotter:2015}.
In this setting, the Bayesian likelihood involves a solution of the mathematical model for the dynamics, which is typically an ODE or time-dependent PDE;
we focus here on the ODE situation.
Even when the governing ODE is deterministic, it may be profitable to perform a probabilistic numerical solution:
possible motivations for doing so include the representation of model error (model inadequacy) in the ODE itself, and the impact of discretisation uncertainty.
When such a probabilistic solver is used for the ODE, the likelihood becomes random in the sense considered in this paper.

Random approximate solution of deterministic ODEs is an old idea \citep{Diaconis:1988, Skilling:1992} that has received renewed attention in recent years \citep{Conrad:2016, HennigOsborneGirolami:2015, LieStuartSullivan:2017, Schober:2014}.
As random forward models, these probabilistic ODE solvers are amenable to the analysis of Section~\ref{sec:general}.
Let $f \colon \Reals^d \to \Reals^d$ and consider the following parameter-dependent initial value problem for a fixed, parameter-independent duration $T>0$:
\begin{align}
	\label{eq:the_ivp}
	\frac{\rd}{\rd t} z(t;u) & = f(z(t;u);u), & \text{for $0\leq t \leq T$,} \\
	\notag
	z(0;u) & = z_{0}(u).
\end{align}

In the context of the BIP presented in Section~\ref{sec:setup}, the unknown parameter $u$ will appear in the definition of the initial condition $z_0 = z_0(u)$ or the right-hand side $f(z(t)) = f(z(t); u)$, resulting in the parameter-dependent solution $(z(t; u))_{t\in[0,T]}$.
Define the solution operator
\begin{equation}
	\label{eq:deterministic_solution_operator}
	S \colon \UU \to C([0,T];\Reals^{d}),\quad u\mapsto S(u) \defeq (z(t;u))_{t\in [0,T]} ,
\end{equation}
where $(z(t;u))_{t\in [0,T]}$ solves \eqref{eq:the_ivp}.
We equip $C([0,T];\Reals^d)$ with the supremum norm.

For notational convenience, we will for the majority of this section not indicate the dependence of $z_0$ or $f$ on $u$.
We will, however, explicitly track the dependence on $z_0$ and $f$ of the error analysis below.

Let $\odeflow^{t} \colon \Reals^{d} \to \Reals^{d}$ be the flow map associated to the initial value problem \eqref{eq:the_ivp}, i.e.\ $\odeflow^{t}(z_{0}) \coloneqq z(t; u) = S(u)(t)$.
Fix a time step $\tstep > 0$ such that $N \defeq T / \tstep \in \Naturals$, and a time grid
\begin{equation}
	\label{eq:time_grid}
	t_{k} \defeq k \tstep \text{ for } k \in [N] \defeq \{0, 1, \dotsc, N \}.
\end{equation}
We denote by $z_{k} \defeq z(t_{k}) \equiv \odeflow^{\tstep}(z_{k - 1})$ the value of the exact solution to \eqref{eq:the_ivp} at time $t_{k}$.
We shall sometimes abuse notation and write $[N] = \{0, 1, \dotsc, N - 1 \}$ or $[N] = \{ 1, 2, \dotsc, N \}$.

To a single-step numerical integration method (e.g.\ a Runge--Kutta method of some order) we shall associate a numerical flow map $\numflow^{\tstep} \colon \Reals^{d} \to \Reals^{d}$.
The numerical flow map approximates the sequence $(z_{k})_{k\in [N]}$ by a sequence $(Z'_{k})_{k\in [N]}$, where $Z'_{k} \defeq \numflow^{\tstep}(Z'_{k - 1})$.
A fundamental task in numerical analysis is to determine sufficient conditions for convergence of the sequence $(Z'_{k})_{k\in [N]}$ to $(z_{k})_{k\in [N]}$.
The investigations of \citet{Conrad:2016} and \citet{LieStuartSullivan:2017} concern a similar task in the context of uncertainty quantification.
Given $\tstep>0$, consider a collection $(\xi_k)_{k\in[N]}$ of stochastic processes $\xi_k \colon \Omega\times [0,\tstep] \to \Reals^{d}$ having almost-surely continuous paths.
Define a stochastic process $(Z_t)_{t\in [0,T]}$ in terms of a new randomised integrator
\begin{equation}
	\label{eq:randomised_numerical_integrator}
	Z(t_{k+1};u) \defeq \numflow^{\tau}(Z(t_{k};u)) + \xi_{k}(\tau).
\end{equation}
The stochastic processes $(\xi_{k})_{k\in [N]}$ are intended to capture the effect of uncertainties, e.g.\ those that arise due to properties of the vector field that are not resolved by the time grid \eqref{eq:time_grid} associated to the time step $\tstep$.
We extend the definition \eqref{eq:randomised_numerical_integrator} to continuous time via
\begin{equation}
	\label{eq:randomised_numerical_integrator_continuous_time}
	Z(t;u)\defeq \numflow^{t-t_k}(Z(t_k;u))+\xi_k(t-t_k),\quad\text{for } t_k<t<t_{k+1}.
\end{equation}
We shall use the $(\xi_k)_{k\in[N]}$ to construct our random approximations to $\Misfit$.
Note therefore that, in order to be consistent with our assumption (see the third paragraph of Section~\ref{sec:general}) that the randomness in the approximation of $\Misfit$ is independent of the randomness in the parameter $u$ being inferred, we shall assume that the $(\xi_k)_{k\in [N]}$ do not depend on the parameter $u$.
However, the map $\numflow^{\tstep}$ does depend on the parameter $u\in\UU$, because $\numflow^{\tstep}$ involves the vector field $f(\quark;u)$.

Define the random solution operator associated to the randomised integrator \eqref{eq:randomised_numerical_integrator_continuous_time}:
\begin{equation}
	\label{eq:random_solution_operator}
	 S_{N} \colon \UU \to C([0,T];\Reals^{d}),\quad u \mapsto S_{N}(u)\defeq (Z(t;u))_{t\in[0,T]},
\end{equation}
where $(Z(t;u))_{t\in[0,T]}$ satisfies \eqref{eq:randomised_numerical_integrator_continuous_time}, and is almost surely continuous.

Let $T_{J}\subset [0,T]$ be a strictly increasing sequence of time points, indexed by a finite, nonempty index set $J$ with cardinality $\absval{ J } \in \Naturals$.
Note that $T_{J}$ may coincide with the time grid defined in \eqref{eq:time_grid};
to increase the scope of the subsequent analysis however, we allow for $T_{J}$ to differ from \eqref{eq:time_grid}.
Let $\YY \defeq \Reals^{d \absval{ J }}$, and equip it with the topology induced by the standard Euclidean inner product.
Define the observation operator 
\begin{equation}
	\label{eq:observation_operator}
	O \colon C([0,T];\Reals^{d}) \to \YY,\quad \tilde{z}\mapsto O\left(\tilde{z}\right) \defeq (\tilde{z}(t_j))_{t_j\in T_J}, 
\end{equation}
which projects some $\tilde{z}\in C([0,T];\Reals^{d})$ to a finite-dimensional vector in $\YY$ constructed by stacking the $\Reals^d$-valued vectors that result from evaluating $\tilde{z}$ at the time points in $T_J$.
We take the norm on $\YY$ to be $\norm{\quark}_{\ell^{d\absval{ J }}_2}$.

Given the operators $S$, $O$, and $S_{N}$ defined in \eqref{eq:deterministic_solution_operator}, \eqref{eq:observation_operator}, and \eqref{eq:random_solution_operator}, we define the forward operators $G,G_{N} \colon \UU \to \YY$ by
\begin{equation}
	\label{eq:forward_operators}
	G\defeq O\circ S,\quad G_{N}\defeq O\circ S_{N}.
\end{equation}
The associated likelihoods are the quadratic misfits given by \eqref{eq:quadratic_misfits} with some fixed, positive-definite matrix $\Gamma$.

We define the continuous-time error process by 
\begin{equation}
	\label{eq:error_continuous_time}
	e(t;u)\defeq z(t;u)-Z(t;u),\quad 0\leq t\leq T.
\end{equation}
Since $T_J$ is a proper subset of $[0,T]$, it follows that
\begin{equation}
	\label{eq:error_in_forward_operators}
	\norm{ G_{N}(u) - G(u) } \leq \absval{J}\sup_{0\leq t\leq T} \norm{e(t;u)}_{\ell^d_2}.
\end{equation}
This completes our formulation of the probabilistic numerical integration of the ODE \eqref{eq:the_ivp} as a random likelihood model of the type considered in Section~\ref{sec:general}.

\subsection{Convergence in continuous time for Lipschitz flows}
\label{ssec:pnode-lipschitz}

In this section, we quote some assumptions and results from \citet{LieStuartSullivan:2017}.
The vector field $f$ in \eqref{eq:the_ivp} induces a flow $\odeflow^{\tstep} \colon \Reals^{d} \to \Reals^{d}$ by 
\begin{equation}
	\label{eq:flow_map}
	\odeflow^{\tstep}(a) = a + \int_0^{\tstep} f(\odeflow^t(a)) \, \rd t.
\end{equation} 

\begin{assumption}[Assumption~3.1, \citet{LieStuartSullivan:2017}]
	\label{ass:exact_flow}
	The vector field $f$ admits $0<\tstep^\ast\leq 1$ and $C_\odeflow\geq 1$, such that for $0 < \tstep < \tstep^\ast$, the flow $\odeflow^{\tstep} \colon \Reals^{d} \to \Reals^{d}$ defined by \eqref{eq:flow_map} is globally Lipschitz, with
	\begin{equation*}
		\norm{\odeflow^{\tstep}(z_0)-\odeflow^{\tstep}(v_0)} \leq (1+C_\odeflow \tstep) \norm{ z_0 - v_0 },\quad\text{for all $z_0, v_0\in\Reals^{d}$.}
	\end{equation*}
\end{assumption}

A globally Lipschitz vector field $f$ in \eqref{eq:the_ivp} yields a flow map $\odeflow^t$ that satisfies Assumption~\ref{ass:exact_flow}.
However, vector fields that satisfy a one-sided Lipschitz condition also have the same property.
Such vector fields have been studied in the numerical analysis literature for both ordinary and stochastic differential equations in the last four decades; see, e.g.\ \citep{Butcher:1975}, and the references cited in Section 3.1 of \citep{HighamStuartMao:2002}.

Recall that $\numflow^{\tstep} \colon \Reals^{d} \to \Reals^{d}$ represents the numerical method that we use to integrate 
\eqref{eq:the_ivp}.

\begin{assumption}[Assumption~3.2, \citet{LieStuartSullivan:2017}]
	\label{ass:numerical_flow}
	The numerical method $\numflow^{\tstep}$ has uniform local truncation error of order $q + 1$:
	for some constant $C_\numflow \geq 1$ that does not depend on $\tstep$,
	\[
		\sup_{v \in \Reals^{d}} \norm{ \numflow^{\tstep}(v) - \odeflow^{\tstep}(v) } \leq C_\numflow \tstep^{q + 1}.
	\]
\end{assumption}

The assumption above is satisfied for both single-step and multistep numerical methods that are obtained by considering vector fields in $C^q(\Reals^d)$, provided that the $q^{\text{th}}$ derivatives are bounded; see Section III.2 of \citep{HairerNorsettWanner:2009}.
We emphasise that the above assumption is made to simplify the analysis, and that the convergence results below extend to the case where the uniform bound does not hold; see Section 4 of \citep{LieStuartSullivan:2017}.

Now recall the collection $(\xi_k(\tstep))_{k\in[N]}$ of random variables, where $\xi_k(\tstep)$ is used in \eqref{eq:randomised_numerical_integrator}.

\begin{assumption}[Assumption~5.1, \citet{LieStuartSullivan:2017}]
	\label{ass:p-R_regularity_condition_on_noise}
	The stochastic processes $(\xi_{k})_{k\in\mathbb{N}}$ admit $p \geq 1$, $R\in \Naturals\cup\{+\infty\}$, and $C_{\xi,R}\geq 1$, independent of $k$ and $\tstep$, such that for all $1\leq r\leq R$ and all $k\in \Naturals$,
	\begin{equation*}
		\EE_{\nu_{N}} \left[\sup_{0 < t \leq T/N} \norm{\xi_{k}(t)}^{r} \right] \leq \left(C_{\xi,R}\left(\frac{T}{N}\right)^{p+1/2}\right)^r.
	\end{equation*}	
\end{assumption}
(In this section, $\nu_{N}$ plays the role of the distribution of the $\xi_{k}$'s.)
The assumption above quantifies the regularity of the $(\xi_k)_{k\in[N]}$ by specifying how many moments each $\xi_k(t)$ has and how quickly these decay with $\tstep=T/N$.
We do not require the $(\xi_k(t))_{k\in[N]}$ to have zero mean, to be independent, or to be identically distributed.
It is shown in Section 5.2 of \citep{LieStuartSullivan:2017} that, for example, the integrated Brownian motion process satisfies Assumption~\ref{ass:p-R_regularity_condition_on_noise}.
The integrated Brownian motion process has been used as a state-independent model of the uncertainty in the off-grid behaviour of solutions to ODEs in \citep{Conrad:2016, Schober:2014, Chkrebtii:2016}.

We now consider the following convergence theorem:

\begin{theorem}[Theorem~5.2, \citet{LieStuartSullivan:2017}]
	\label{thm:strong_error_sup_inside_continuous}
	Suppose that $e_0=0$, and suppose that Assumptions~\ref{ass:exact_flow}, \ref{ass:numerical_flow}, and \ref{ass:p-R_regularity_condition_on_noise} hold with parameters $\tstep^\ast$, $C_\odeflow$, $C_\numflow$, $q$, $C_{\xi,R}$, $p$, and $R$.
	Let $n\in\Naturals$, with $n\leq R$.
	Then, for all $T/\tstep^\ast<N$,
	\begin{align}
		\label{eq:strong_error_sup_inside_continuous}
		\EE_{\nu_{N}}\left[\sup_{0\leq t\leq T}\norm{e(t;u)}^n\right]\leq 3^{n-1}\left(\left(1+C_\odeflow \tstep^\ast\right)^n\overline{C}+C_\numflow^n(\tstep^\ast)^n+TC^n_{\xi,R}\right)\left(\frac{T}{N}\right)^{n(q\wedge (p-1/2))},
	\end{align}
	where 
	\begin{align*}
		\overline{C}&\defeq 2T\max\{(4C_\numflow)^n,(2 C_{\xi,R})^n\}\exp\left(TC_\odeflow(n,\tstep^\ast)\right)
		\\
		C_\odeflow(n,\tstep^\ast)&\defeq \left[(1+\tstep^\ast 2^{n-1})^2(1+\tstep^\ast C_\odeflow)^n-1\right](\tstep^\ast)^{-1}.
	\end{align*} 
\end{theorem}
Note that the scalars $\overline{C}$ and $C_\odeflow(n,\tstep^\ast)$ depend on $u\in \UU$, since $C_\odeflow$ and $C_\numflow$ depend on the vector field $f$, which in turn depends on the parameter $u$.

Recall that the random variable $(Z(t;u))_{0\leq t\leq T}$ defined in \eqref{eq:randomised_numerical_integrator_continuous_time} is a random surrogate for the true solution of the ODE \eqref{eq:the_ivp}.
The law of $(Z(t;u))_{0\leq t\leq T}$ is thus a probability measure on the space of continuous paths defined on the interval $[0,T]$.
With this in mind, the interpretation of Theorem~\ref{thm:strong_error_sup_inside_continuous} is that the law of $(Z(t;u))_{0\leq t\leq T}$ contracts to the Dirac distribution located at the true solution $(z(t;u))_{0\leq t\leq T}$ of \eqref{eq:the_ivp}, as the spacing $T/N$ in the time grid \eqref{eq:time_grid} decreases to zero.
Equivalently, given the true solution operator $S$ and its random counterpart $S_{N}$, Theorem~\ref{thm:strong_error_sup_inside_continuous} implies that the random solution operator converges in the $L^{n}$ topology to the true solution operator.
Thus, Theorem~\ref{thm:strong_error_sup_inside_continuous} guarantees that by refining the time grid, one reduces the uncertainty over the solution of \eqref{eq:the_ivp}.
This is a desirable feature for uncertainty quantification, since estimates of the solution uncertainty can also be fed forward to obtain estimates of the uncertainty of \textit{functionals} of the solution, and since the probabilistic description allows for a more nuanced description of the uncertainty compared to the usual worst-case description that is common in the numerical analysis of deterministic methods.

\begin{corollary}[Corollary~5.3, \cite{LieStuartSullivan:2017}]
	\label{cor:finite_mgf_lipschitz_flow_continuous_time}
	Fix $n\in \Naturals$.
	Suppose that Assumptions~\ref{ass:exact_flow} and \ref{ass:numerical_flow} hold, and that Assumption \ref{ass:p-R_regularity_condition_on_noise} holds with $R=+\infty$ and $p\geq 1/2$.
	Then, for all $0 < \tstep < \tstep^\ast$,
	\begin{equation}
		\EE_{\nu_{N}}\left[\exp\left(\rho\sup_{0\leq t\leq T}\norm{e(t)}^n\right)\right]<\infty,\quad\text{for all }\rho\in\Reals.
	\end{equation}
\end{corollary}
Since the exponential integrability of a random variable is related to the exponential concentration of its values about its mean or median, the above result shows that strong assumptions on the model of uncertainty translate to strong conclusions about the behaviour of the corresponding error.
In the context of random approximations of BIPs, we shall use Corollary~\ref{cor:finite_mgf_lipschitz_flow_continuous_time} in order to establish the convergence of the random approximations in the Hellinger sense.

\subsection{Effect of probabilistic integration on Bayesian posterior distribution}

Define the approximate posteriors $\mu^{\marginal}_{N}$ and $\mu^{\sample}_{N}$ according to \eqref{eq:marginal_posterior} and \eqref{eq:random_posterior}, using the quadratic misfits $\Misfit$ and $\Misfit_{N}$ from \eqref{eq:quadratic_misfits} and the forward models $G=O\circ S$ and $G_{N}=O\circ S_{N}$ given in \eqref{eq:deterministic_solution_operator} and \eqref{eq:random_solution_operator} respectively, where $O$ denotes the observation operator associated to a fixed, finite sequence $T_J$ of observation times in $[0,T]$.

As we saw in the last section, the results of \citet{LieStuartSullivan:2017} guarantee convergence in the $L^n$ topology of the random solution operator $S_{N}$ to the true solution operator $S$.
It is of interest to determine whether one can use this result to guarantee that one can perform inference over $u$ using the probabilistic integrator.
In particular, given that the probabilistic integrator provides a random approximation, it is of interest to determine whether one can obtain results that are not only reasonable, but that improve as the time resolution of the probabilistic integrator increases.
The following result shows that this is indeed the case:
as the time resolution increases, the random forward model $G_{N}$ yields a random posterior over the parameter space that converges in the Hellinger topology to the true posterior at the expected rate.


\begin{theorem}
	\label{thm:convergence_misfit_pn_ode}
	Suppose that $\UU$ is a compact subset of $\Reals^{m}$ for some $m\in \Naturals$, and suppose that $S,S_{N} \colon \UU \to C([0,T];\Reals^{d})$ are continuous maps.
	Let $2 < \rho^\ast < \infty$ be arbitrary.
	Suppose that $e_0=0$, and that Assumptions~\ref{ass:exact_flow}, \ref{ass:numerical_flow}, and \ref{ass:p-R_regularity_condition_on_noise} hold with parameters $\tstep^\ast$, $C_\odeflow$, $C_\numflow$, $q$, $R=+\infty$, $C_{\xi,R}$, and $p$, and that these parameters depend continuously on $u$.
	Then, for $N\in \Naturals$ such that $T/\tstep^\ast<N$, the following hold:
	\begin{compactenum}[(a)]
	\item there exists some $C>0$ that does not depend on $N$, such that \eqref{eq:hell_conv_marg_result_forward_model} holds for $r_1=1$ and $r_2=2\rho^\ast/(\rho^\ast-1)$, and
	\item there exists some $D>0$ that does not depend on $N$, such that \eqref{eq:hell_conv_rand_result_forward_model} holds for $s_1=2\rho^\ast/(\rho^\ast-2)$ and $s_2=2$.
	\end{compactenum}
\end{theorem}

The parameter $\rho^{\ast}$ above plays the same role as $\rho^{\ast}$ in Theorem~\ref{thm:cvgce_of_random_measures}, case (\ref{thm:cvgce_of_random_measures_2});
$\rho^{\ast}$ quantifies the exponential decay of the random misfit $\Misfit_{N}$ with respect to $\nu_{N}$.
For this reason, $\rho^{\ast}$ is constrained to the same range of values $2<\rho^{\ast}<\infty$ as given there, and determines the parameters $r_{2}$ and $s_{1}$ that partly describe the convergence rates in \eqref{eq:hell_conv_marg_result_forward_model} and \eqref{eq:hell_conv_rand_result_forward_model}.
As shown in the proof, the reason why $\rho^{\ast}$ does not appear to play any further role is due to \eqref{eq:error_in_forward_operators} and Corollary~\ref{cor:finite_mgf_lipschitz_flow_continuous_time}.
In particular, since Assumption~\ref{ass:p-R_regularity_condition_on_noise} holds with $R=+\infty$, Corollary~\ref{cor:finite_mgf_lipschitz_flow_continuous_time} ensures that the exponential decay parameter $\rho^{\ast}$ need not be constrained to any bounded interval.

The continuous dependence on $u$ of the parameters of Assumptions~\ref{ass:exact_flow}, \ref{ass:numerical_flow} and \ref{ass:p-R_regularity_condition_on_noise} also allows for parameters that do not depend on $u$, e.g.\ $R=+\infty$.
The assumption of continuous dependence on $u$ of the parameters, and the assumptions on $\UU$, ensure that the map $u\mapsto \EE_{\nu_{N}}[\exp(\rho^{\ast}\Misfit_{N}(u))]$ is uniformly bounded by a scalar that depends only on $\UU$; from this the exponential integrability hypothesis on $\Misfit_{N}$ of Theorem~\ref{thm:cvgce_of_random_measures}(\ref{thm:cvgce_of_random_measures_2}) holds, and we can apply the corresponding conclusions.
While these assumptions may appear to be  strong, they simplify the analysis considerably and thus are not uncommon in the literature on parameter inference for dynamical systems.
We leave the investigation of weaker assumptions for future work.

\section{Concluding remarks}
\label{sec:conclusions}

In this paper we have considered the impact upon a BIP of replacing the log-likelihood function $\Misfit$ by a random function $\Misfit_{N}$.
Such approximations occur for example when a cheap stochastic emulator is used in place of an expensive exact log-likelihood, or when a probabilistic solver is used to simulate the forward model.

Our results show that such approximations are well-posed, with the approximate Bayesian posterior distribution converging to the true Bayesian posterior as the error between $\Misfit$ and $\Misfit_N$, measured in a suitable sense, goes to zero.
More precisely, we have shown that the convergence rate of the random log-likelihood $\Misfit_{N}$ to $\Misfit$ --- as assessed in a nested $L^{p}$ norm with respect to the distribution $\nu_{N}$ of $\Misfit_N$ and the Bayesian prior distribution $\mu_{0}$ of the unknown $u$ --- transfers to convergence of two natural approximations to the exact Bayesian posterior $\mu$, namely (a) the randomised posterior measure $\mu_{N}^{\sample}$ that simply has $\Misfit_{N}$ in place of $\Misfit$, and (b) the deterministic pseudo-marginal posterior measure $\mu_{N}^{\marginal}$, in which the likelihood function and marginal likelihood of $\mu_{N}^{\sample}$ are individually averaged with respect to $\nu_{N}$.

Since the hypotheses that are required for these results operate directly at the level of finite-order moments of the error $\Misfit_{N} - \Misfit$, the convergence results in this paper automatically apply to GP approximations, as previously considered by \citet{StuartTeckentrup:2017}, which have moments of all orders.
However, in a substantial generalisation, Theorems~\ref{thm:hell_conv_marg} and \ref{thm:hell_conv_rand} show that, in the $L^{2}$ case,
\begin{align*}
	\dhh \bigl( \mu, \mu_{N}^{\marginal} \bigr) &\leq C \; \Norm{ \EE_{\nu_{N}} \bigl[ \absval{ \Misfit - \Misfit_{N} } \bigr] }_{L^{2}_{\mu_{0}}(\UU)}, \\
	\EE_{\nu_{N}} \left[ \dhh \bigl( \mu, \mu_{N}^{\sample} \bigr)^2 \right]^{1/2} &\leq C \; \Norm{ \EE_{\nu_{N}} \left[ \absval{ \Misfit - \Misfit_{N} }^{2} \right]^{1/2} }_{L^{2}_{\mu_{0}}(\UU)},
\end{align*}
for general approximations $\Misfit_N$.
This optimal bound requires that the random misfit $\Misfit_N$ allows pointwise bounds on $\exp(-\Misfit_N)$ and $Z_N^\sample$ with respect to the distribution $\nu_{N}$ on $\Misfit_N$ and the Bayesian prior $\mu_{0}$ on the unknown $u$.
If the distribution of $\Misfit_N$ does not allow pointwise ($L^\infty$) bounds on $\exp(-\Misfit_N)$ and $Z_N^\sample$, but only bounds in $L^r$ for some $1 \leq r < \infty$, then the norms in the bounds above need to be strengthened to higher order $L^q_{\mu_0}(\UU)$ norms and/or higher order moments of the error $\Misfit - \Misfit_N$, resulting in the quantity $\Norm{ \EE_{\nu_{N}} \left[ \absval{ \Misfit - \Misfit_{N} }^{p} \right]^{1/p} }_{L^{q}_{\mu_{0}}(\UU)}$ appearing on the right hand sides, for some $q \geq 2$ and $p \geq 1$ (respectively $p \geq 2$).

Our error bounds are explicit in the sense that the aforementioned exponents $p$ and $q$ can typically be calculated explicitly given the structure of $\Misfit_N$.
This is the case for the GP emulators considered in \citet{StuartTeckentrup:2017}, and also the randomised misfit models and probabilistic numerical solvers considered here.
The constant $C$ in the error bounds, on the other hand, is typically not computable in advance;
it involves quantities such as the normalising constants $Z$ and $\EE[Z_N^\sample]$, which for most forward models $G$ are not known analytically and very expensive to compute numerically.
In a sense, this is similar to the everyday situation of using an ODE or PDE solver of known order but unknown constant prefactor.

A significant open question in this work is the one highlighted at the end of Section~\ref{sec:random-misfit}:
in contrast to randomised dimension reduction using bounded random variables, is the case of \emph{Gaussian} randomly-projected misfits one in which the MAP problem and BIP genuinely have different convergence properties?

\section*{Acknowledgements}
\addcontentsline{toc}{section}{Acknowledgements}

ALT is partially supported by The Alan Turing Institute under the EPSRC grant EP/N510129/1.
HCL and TJS are partially supported by the Freie Universit\"{a}t Berlin within the Excellence Initiative of the German Research Foundation (DFG).
This work was partially supported by the DFG through grant CRC 1114 ``Scaling Cascades in Complex Systems'', and by the National Science Foundation (NSF) under grant DMS-1127914 to the Statistical and Applied Mathematical Sciences Institute (SAMSI) and SAMSI's QMC Working Group II ``Probabilistic Numerics''.
Any opinions, findings, and conclusions or recommendations expressed in this article are those of the authors and do not necessarily reflect the views of the above-named funding agencies and institutions.

\bibliographystyle{abbrvnat}
\addcontentsline{toc}{section}{References}
\bibliography{references.bib}

\section*{Appendix: Proofs of Results}
\addcontentsline{toc}{section}{Appendix: Proofs of Results}
\setcounter{equation}{0}
\renewcommand{\theequation}{A.\arabic{equation}}

The proofs in this section will make repeated use of the following inequalities for real $a$ and $b$:
\begin{align}
	\label{eq:Young_plus_Cauchy-Schwarz}
	(a - b)^{2} & \leq 2 a^{2} + 2 b^{2}, \\
	\label{eq:diff_sq_bound}
	(a-b)^{2} = \left( \frac{a^{2}-b^{2}}{a+b} \right)^{2} & \leq \frac{(a^{2}-b^{2})^{2}}{a^{2}+b^{2}} , \\
	\label{eq:exp_Lipschitz}
	\absval{ \exp(a) - \exp(b) } & \leq (\exp(a) + \exp(b)) \absval{ a - b } , \\
	\label{eq:bound_on_1_over_x_plus_y_times_xy}
	[ ( a + b ) a b ]^{-1} & \leq \max \{ a^{-3}, b^{-3} \} & & \text{for $a, b > 0$.}
\end{align}
We also have, for arbitrary $N \in \Naturals$ and $p\geq 1$ (not necessarily integer-valued), by the triangle inequality and Jensen's inequality,
\begin{align}
	\Absval{ \sum^N_{j=1}s_j }^p & \leq N^p\left(\frac{1}{N}\sum^N_{j=1} \absval{ s_{j} }\right)^p \leq N^{p-1}\sum_{j=1}^N \absval{ s_{j} }^p,
	\label{eq:gen_triangle_inequality}
\end{align}

\begin{proof}[Proof of Theorem~\ref{thm:hell_conv_marg}]
	Using \eqref{eq:rad_nik} and \eqref{eq:random_posterior}, we have
	\begin{align*}
		\sqrt{\frac{\rd \mu}{\rd \mu_{0}}} - \sqrt{\frac{\rd \mu_{N}^{\marginal}}{\rd \mu_{0}}}
		& = \frac{\sqrt{\exp(-\Misfit(u))}}{Z^{1/2}}-\frac{\sqrt{\EE_{\nu_{N}} \bigl[ \exp ( - \Misfit_{N}(u) ) \bigr]}}{\EE_{\nu_{N}} \bigl[ Z_{N}^{\sample} \bigr]^{1/2}}
		\\
		& = \frac{\sqrt{\exp(-\Misfit(u))}-\sqrt{\EE_{\nu_{N}} \bigl[ \exp ( - \Misfit_{N}(u) ) \bigr]}}{Z^{1/2}} \\ 
		& \phantom{=} \quad + \sqrt{\EE_{\nu_{N}} \bigl[ \exp ( - \Misfit_{N}(u) ) \bigr]}\left(\frac{1}{Z^{1/2}}-\frac{1}{\EE_{\nu_{N}} \bigl[ Z_{N}^{\sample} \bigr]^{1/2}}\right).
	\end{align*}
	Inequality \eqref{eq:Young_plus_Cauchy-Schwarz} with $a = Z^{-1/2} \bigl( e^{-\Misfit(u) / 2} - \EE_{\nu_{N}}[ e^{- \Misfit_{N}(u)} ]^{1/2} \bigr)$ and $b = \EE_{\nu_{N}} [ Z_{N}^{\sample} ]^{1/2}(Z^{-1/2}-\EE_{\nu_{N}}[ Z_{N}^{\sample} ]^{-1/2})$ and the definition \eqref{eq:Hellinger} of the Hellinger distance $\dhh$ yield
	\begin{align*}
		2 \; \dhh \bigl( \mu, \mu_{N}^{\marginal} \bigr)^{2} 
		& = \int_\UU \left( \sqrt{\frac{\rd \mu}{\rd \mu_{0}}}(u) - \sqrt{\frac{\rd \mu_{N}^{\marginal}}{\rd \mu_{0}}}(u) \right)^{2} \, \rd \mu_{0} (u) \\
		&\leq \frac{2}{Z} \Norm{ \left(\sqrt{\exp ( - \Misfit ) } - \sqrt{\EE_{\nu_{N}} \bigl[ \exp ( - \Misfit_{N} ) \bigr]}\right)^{2} }_{L^1_{\mu_{0}}} \\
		& \qquad + 2 \EE_{\nu_{N}} \bigl[ Z_{N}^{\sample} \bigr] \left(Z^{-1/2} - \EE_{\nu_{N}} \bigl[ Z_{N}^{\sample} \bigr]^{-1/2} \right)^{2} \qefed I + II.
	\end{align*}

	For the first term, we use inequality \eqref{eq:diff_sq_bound} with $a = e^{- \Misfit(u) / 2}$ and $b = \EE_{\nu_{N}} [ \exp ( - \Misfit_{N}(u) )]^{1/2}$, together with H\"older's inequality with conjugate exponents $p_1$ and $p_1'$, to derive
	\begin{align}
		\frac{Z}{2} I &\leq \Norm{ \left( \exp ( - \Misfit) - \EE_{\nu_{N}} \bigl[ \exp ( - \Misfit_{N}) \bigr] \right)^{2} \left(\exp ( - \Misfit ) + \EE_{\nu_{N}} \bigl[ \exp ( - \Misfit_{N} ) \bigr]\right)^{-1} }_{L^1_{\mu_{0}}} \notag\\
		& \leq \Norm{ \left(\exp ( - \Misfit ) - \EE_{\nu_{N}} \bigl[ \exp ( - \Misfit_{N} ) \bigr] \right)^{2} }_{L^{p_1'}_{\mu_{0}}} \Norm{ \left(\exp ( - \Misfit ) + \EE_{\nu_{N}} \bigl[ \exp ( - \Misfit_{N} ) \bigr] \right)^{-1} }_{L^{p_1}_{\mu_{0}}}.		
		\label{eq:a00}
	\end{align}
	We estimate the second factor on the right-hand side of \eqref{eq:a00}.
	Using the facts that $x \mapsto 1/x$ is decreasing on $(0,\infty)$, that $(x+y)^{-1}\leq \min\{ x^{-1},y^{-1}\}$ for all $x,y>0$, and that both $\exp ( - \Misfit(u) )$ and $\EE_{\nu_{N}}[ \exp ( - \Misfit_{N}(u) )]$ are strictly positive, we obtain
	\begin{equation*}
		\Norm{ \left(\exp ( - \Misfit ) + \EE_{\nu_{N}} \bigl[ \exp ( - \Misfit_{N} ) \bigr] \right)^{-1} }_{L^{p_1}_{\mu_{0}}} \leq \Norm{ \min\left\{ \exp ( \Misfit ),\EE_{\nu_{N}} \bigl[ \exp ( - \Misfit_{N} ) \bigr]^{-1} \right\} }_{L^{p_1}_{\mu_{0}}}.
	\end{equation*}
	For $f,g\in L^{1}_{\mu_{0}}(\UU)$, the partition $\UU=\{ f<g\} \uplus \{f\geq g\}$ and the corresponding integral inequalities on $\{f<g\}$ and $\{f \geq g\}$ imply that $\norm{ \min\{ f, g \} }_{L^1_{\mu_{0}}} \leq \min\{ \norm{ f }_{L^1_{\mu_{0}}}, \norm{ g }_{L^1_{\mu_{0}}}\}$.
	Hence,
	\begin{equation}
		\label{eq:a03}
		\Norm{ \min\left\{ e^{-\Misfit} , \EE_{\nu_{N}} \bigl[ e^{- \Misfit_{N}} \bigr]^{-1} \right\} }_{L^{p_1}_{\mu_{0}}}\leq \min\left\{ \norm{ e^{\Misfit} }_{L^{p_1}_{\mu_{0}}}, \Norm{ \EE_{\nu_{N}} \bigl[ e^{- \Misfit_{N}} \bigr]^{-1} }_{L^{p_1}_{\mu_{0}}} \right\} \leq C_1,
	\end{equation}
	where $C_1=C_1(p_1)$ is the constant specified in assumption (\ref{item:hell_conv_marg_a}).
	This completes our estimate for the second factor on the right-hand side of \eqref{eq:a00}.
	For the first factor, the linearity of expectation, inequality \eqref{eq:exp_Lipschitz}, and H\"older's inequality with conjugate exponents $p_2, p_2'$ with respect to $\nu_{N}$ and $p_3, p_3'$ with respect to $\mu_{0}$ give
	\begin{align}
		&\Norm{ \left( \exp ( - \Misfit ) - \EE_{\nu_{N}} \bigl[ \exp ( - \Misfit_{N} ) \bigr] \right)^{2} }_{L^{p_1'}_{\mu_{0}}} = \Norm{ \EE_{\nu_{N}} \bigl[ \exp ( - \Misfit ) - \exp ( - \Misfit_{N} ) \bigr] ^{2} }_{L^{p_1'}_{\mu_{0}}}
		\notag\\
		&\qquad \leq \Norm{ \EE_{\nu_{N}} \bigl[ \absval{ \exp ( - \Misfit ) + \exp ( - \Misfit_{N} ) } \absval{ \Misfit - \Misfit_{N} } \bigr]^{2} }_{L^{p_1'}_{\mu_{0}}}
		\notag\\
		&\qquad \leq \Norm{ \EE_{\nu_{N}} \bigl[ \big(\exp ( - \Misfit ) + \exp ( - \Misfit_{N}) \big)^{p_2} \bigr]^{2/p_2} \EE_{\nu_{N}} \bigl[ \absval{ \Misfit - \Misfit_{N} }^{p_2'} \bigr]^{2/p_2'} }_{L^{p_1'}_{\mu_{0}}}
		\notag\\
		&\qquad \leq \Norm{ \EE_{\nu_{N}} \bigl[ \big(\exp ( - \Misfit ) + \exp ( - \Misfit_{N} ) \big)^{p_2} \bigr]^{1 /p_2} }^{2}_{L^{2p_1'p_3}_{\mu_{0}}} \Norm{ \EE_{\nu_{N}} \bigl[ \absval{ \Misfit - \Misfit_{N} }^{p_2'} \bigr]^{1/p_2'} }^{2}_{L^{2p_1' p_3'}_{\mu_{0}}}
		\label{eq:a04}
	\end{align}
	Letting $C_2=C_2(p_1',p_2,p_3)$ be the constant in assumption (\ref{item:hell_conv_marg_b}), and using \eqref{eq:a03}, it follows that 
	\[
		I \leq \frac{2}{Z}\cdot C_1(p_1)\cdot C^{2}_2(p_1',p_2,p_3)\cdot \Norm{ \EE_{\nu_{N}} \bigl[ \absval{ \Misfit - \Misfit_{N} }^{p_2'} \bigr]^{1/p_2'} }^{2}_{L^{2p_1' p_3'}_{\mu_{0}}}.
	\]

	Now inequality \eqref{eq:diff_sq_bound} with $a=\EE_{\nu_{N}}[Z^{\sample}_N]^{-1/2}$ and $b=Z^{-1/2}$ and inequality \eqref{eq:bound_on_1_over_x_plus_y_times_xy} yield
	\begin{align*}
		\frac{1}{2 \EE_{\nu_{N}} \bigl[ Z_{N}^{\sample} \bigr]} II
		& = \left(Z^{-1/2} - \big(\EE_{\nu_{N}} \bigl[ Z_{N}^{\sample} \bigr] \big)^{-1/2} \right)^{2} \\
		& =\left(\frac{ \EE_{\nu_{N}} \bigl[ Z_{N}^{\sample} \bigr]-Z}{Z \EE_{\nu_{N}} \bigl[ Z_{N}^{\sample} \bigr]}\right)^{2} \frac{Z \EE_{\nu_{N}} \bigl[ Z_{N}^{\sample} \bigr]}{Z+ \EE_{\nu_{N}} \bigl[ Z_{N}^{\sample} \bigr]} \\
		& \leq \left(\EE_{\nu_{N}} \bigl[ Z_{N}^{\sample} \bigr]-Z\right)^{2}\max \bigl\{ Z^{-3}, \EE_{\nu_{N}} \bigl[ Z_{N}^{\sample}\bigr]^{-3} \bigr\}
		\\
		& \leq \left(\EE_{\nu_{N}} \bigl[ Z_{N}^{\sample} \bigr]-Z\right)^{2}\max \bigl\{ Z^{-3}, C_3^{-3} \bigr\},
	\end{align*}
	where the last inequality follows from assumption (\ref{item:hell_conv_marg_c}).
	
	Using Tonelli's theorem, Jensen's inequality, inequality \eqref{eq:exp_Lipschitz}, and H\"{o}lder's inequality with the same conjugate exponent pairs that we used to obtain \eqref{eq:a04},
	\begin{align*}
		& \left(\EE_{\nu_{N}} \bigl[ Z_{N}^{\sample} \bigr]-Z\right)^{2} \\
		& \quad = \EE_{\mu_{0}} \bigl[\EE_{\nu_{N}}\bigl[ \exp(-\Misfit_{N}) - \exp(-\Misfit) \bigr] \bigr]^{2p_1'/p_1'} \\
		& \quad \leq \Norm{ \EE_{\nu_{N}} \bigl[ \exp ( - \Misfit ) - \exp ( - \Misfit_{N} ) \bigr]^{2} }_{L^{p_1'}_{\mu_{0}}}
		\notag\\
		& \quad \leq \Norm{ \EE_{\nu_{N}} \bigl[ \big(\exp ( - \Misfit ) + \exp ( - \Misfit_{N} ) \big)^{p_2} \bigr]^{ 1 /p_2} }^{2}_{L^{2p_1'p_3}_{\mu_{0}}} \Norm{ \EE_{\nu_{N}} \bigl[ \absval{ \Misfit - \Misfit_{N} }^{p'_2} \bigr]^{1/p'_2} }^{2}_{L^{2 p_1'p'_3}_{\mu_{0}}}
		\\
		& \quad \leq C^{2}_2(p_1,p_2,p_3) \Norm{ \EE_{\nu_{N}} \bigl[ \absval{ \Misfit - \Misfit_{N} }^{p'_2} \bigr]^{1/p'_2} }^{2}_{L^{2 p_1'p'_3}_{\mu_{0}}},
	\end{align*}
	where assumption (\ref{item:hell_conv_marg_b}) yields the last inequality.
	Combining the estimates for $I$ and $II$ yields \eqref{eq:hell_conv_marg_result}.
\end{proof}

\begin{proof}[Proof of Theorem~\ref{thm:hell_conv_rand}]
	This proof is similar to the proof of Theorem~\ref{thm:hell_conv_marg}.
	Since
	\begin{equation*}
	 \sqrt{\frac{\rd \mu}{\rd \mu_{0}}} - \sqrt{\frac{\rd \mu_{N}^{\sample}}{\rd \mu_{0}}} = \frac{e^{- \Misfit(u) / 2} - e^{- \Misfit_{N}(u) / 2}}{Z^{1/2}}- e^{- \Misfit_{N}(u) / 2} \left(\frac{1}{\sqrt{Z^{\sample}_N}}-\frac{1}{Z^{1/2}}\right),
	\end{equation*}
	Tonelli's theorem, inequality \eqref{eq:Young_plus_Cauchy-Schwarz}, and Jensen's inequality yield
	\begin{align*}
		\EE_{\nu_{N}} \bigl[ \dhh \bigl( \mu, \mu_{N}^{\sample} \bigr)^{2} \bigr]
		& = \frac{1}{2} \Norm{ \EE_{\nu_{N}} \left[ \left( \sqrt{\frac{\rd \mu}{\rd \mu_{0}}} - \sqrt{\frac{\rd \mu_{N}^{\sample}}{\rd \mu_{0}}} \right)^{2} \right] }_{L^1_{\mu_{0}}} \\
		& \leq \frac{1}{Z} \Norm{ \EE_{\nu_{N}} \left[ \left( \sqrt{\exp (-\Misfit)} - \sqrt{\exp(-\Misfit_{N})}\right)^{2}\right] }_{L^1_{\mu_{0}}} \\
		& \phantom{=} \quad + \EE_{\nu_{N}} \left[ Z_{N}^{\sample} \bigl( Z^{-1/2} - \bigl( Z_{N}^{\sample} \bigr)^{-1/2} \bigr)^{2} \right] \\
		& \qefed I + II.
	\end{align*}
	For the first term $I$, inequality \eqref{eq:exp_Lipschitz}, and H\"older's inequality with conjugate exponent pairs $(q_1, q_1')$ and $(q_2, q_2')$ give
	\begin{align*}
		Z I 
		&= \Norm{ \EE_{\nu_{N}} \left[ \left( \sqrt{\exp (-\Misfit)} - \sqrt{\exp(-\Misfit_{N})}\right)^{2} \right] }_{L^1_{\mu_{0}}} \\
		&\leq \frac{1}{4}\Norm{ \EE_{\nu_{N}} \bigl[ \absval{ \exp ( - \Misfit/2 ) + \exp ( - \Misfit_{N}/2 ) }^{2} \absval{ \Misfit - \Misfit_{N} }^{2} \bigr]}_{L^1_{\mu_{0}}}\\
		&\leq \Norm{ \EE_{\nu_{N}} \bigl[ \absval{ \exp ( - \Misfit/2) + \exp ( - \Misfit_{N}/2) }^{2q_1} \bigr]^{1/q_1} \EE_{\nu_{N}} \bigl[ \absval{ \Misfit - \Misfit_{N} }^{2q_1'} \bigr]^{1/q_1'} }_{L^1_{\mu_{0}}} \\
		&\leq \Norm{ \EE_{\nu_{N}} \bigl[ \big(\exp ( - \Misfit/2 ) + \exp ( - \Misfit_{N}/2 ) \big)^{2q_1} \bigr]^{1/q_1} }_{L^{q_2}_{\mu_{0}}} \Norm{\EE_{\nu_{N}} \bigl[ \absval{ \Misfit - \Misfit_{N} }^{2q_1'} \bigr]^{1/2 q_1'}}^{2}_{L^{2q_2'}_{\mu_{0}}}.
	\end{align*}
	By (\ref{item:hell_conv_rand_a}), we may bound the first factor on the right-hand side of the last inequality by $D_1(q_1,q_2)$.
	Now by \eqref{eq:diff_sq_bound} with $a=Z^{-1/2}$ and $b=(Z^{\sample}_N)^{-1/2}$, and by inequality \eqref{eq:bound_on_1_over_x_plus_y_times_xy}, we obtain (see the proof of Theorem~\ref{thm:hell_conv_marg} after \eqref{eq:a04}) that
	\begin{align*}
		II \leq \EE_{\nu_{N}} \Bigl[ Z_{N}^{\sample} \max \bigl\{ Z^{-3}, \bigl( Z_{N}^{\sample} \bigr)^{-3} \bigr\} \bigl( Z - Z_{N}^{\sample} \bigr)^{2} \Bigr].
	\end{align*}
	Jensen's inequality and another application of inequality \eqref{eq:diff_sq_bound} yield 
	\begin{align*}
		\bigl( Z - Z_{N}^{\sample} \bigr)^{2}
		& \leq \norm{\exp ( - \Misfit) - \exp ( - \Misfit_{N} )}^{2}_{L^{2}_{\mu_{0}}}  \leq \bignorm{ \bigl( \exp ( - \Misfit) + \exp ( - \Misfit_{N}) \bigr)^{2} (\Misfit - \Misfit_{N})^{2} }_{L^1_{\mu_{0}}}.
	\end{align*}
	Combining the preceding two estimates, using Tonelli's theorem and H\"older's inequality with the same conjugate exponent pairs $(q_1, q_1')$ and $(q_2, q_2')$ as used in the bound for $I$, and using (\ref{item:hell_conv_rand_b}), we get
	\begin{align*}
		II 
		& \leq 
		\Norm{\EE_{\nu_{N}} \left[ Z_{N}^{\sample} \max \bigl\{ Z^{-3}, \bigl( Z_{N}^{\sample} \bigr)^{-3} \bigr\} \left( e^{- \Misfit} + e^{- \Misfit_{N}} \right)^{2} (\Misfit - \Misfit_{N})^{2} \right] }_{L^1_{\mu_{0}}} \\
		&\leq \Norm{\EE_{\nu_{N}} \left[ \left(Z_{N}^{\sample} \max \bigl\{ Z^{-3}, \bigl( Z_{N}^{\sample} \bigr)^{-3} \bigr\} \left( e^{- \Misfit} + e^{- \Misfit_{N}} \right)^{2} \right)^{q_1}\right]^{\tfrac{1}{q_1}} \EE_{\nu_{N}} \left[ \absval{ \Misfit - \Misfit_{N} }^{2q_1'} \right]^{\tfrac{1}{q_1'}}}_{L^1_{ \mu_{0}}} \\
		&\leq D_2(q_1,q_2)\Norm{\EE_{\nu_{N}} \bigl[ \absval{ \Misfit - \Misfit_{N} }^{2q_1'} \bigr]^{1/2q_1'}}^{2}_{L^{2q_2'}_{\mu_{0}}} .
	\end{align*}
	Combining the preceding estimates yields \eqref{eq:hell_conv_rand_result}.
\end{proof}

\begin{proof}[Proof of Lemma~\ref{lemma:goodsuffcon_for_hell_conv_thms_for_asmp_set_1}]
	Since $\exp(\Misfit)\in L^{p^\ast}_{\mu_{0}}$, examination of assumption (\ref{item:hell_conv_marg_a}) of Theorem~\ref{thm:hell_conv_marg} indicates that we may set $p_1=p^\ast$ and $C_1\coloneqq \norm{ \exp(\Misfit) }_{L^{p^\ast}_{\mu_{0}}}$.
	By \eqref{eq:lower_bounds_on_potentials_Phi_PhiN}, it follows that $\EE_{\nu_{N}}[\exp(-\Misfit)+\exp(-\Misfit_{N})]\leq 2\exp(C_0)$; thus assumption (\ref{item:hell_conv_marg_b}) of Theorem~\ref{thm:hell_conv_marg} holds with $p_2=p_3=+\infty$ (so that $2p_1'p_3=+\infty$) and $C_2=2\exp(C_0)$.
	We now prove that Assumption (\ref{item:hell_conv_marg_c}) of Theorem~\ref{thm:hell_conv_marg} holds.
	It follows by setting $x=-\Misfit$ and $y=-\Misfit_{N}$ in inequality \eqref{eq:exp_Lipschitz} that $\absval{ \exp(-\Misfit)-\exp(-\Misfit_{N}) } \leq 2\exp(C_0) \absval{ \Misfit-\Misfit_{N} }$.
	Thus
	\begin{align}
		\bigabsval{ Z_{N}^{\sample} - Z }
		& = \bigabsval{ \EE_{\mu_{0}}\bigl[ \exp(-\Misfit_{N}) - \exp(-\Misfit) \bigr] } \notag \\
		& \leq \EE_{\mu_{0}} \bigl[ \absval{ \exp(-\Misfit_{N}) - \exp(-\Misfit) } \bigr]
		\notag
		\\
		& \leq 2\exp(C_0)\EE_{\mu_{0}} \bigl[ \absval{ \Misfit - \Misfit_{N} } \bigr] 
		\label{eq:yarg04}.
	\end{align}
	Using Jensen's inequality, \eqref{eq:yarg04}, Tonelli's theorem, and \eqref{eq:boundedness_of_EmuEnu_N_Phi_N_errorsequence},
	\[
		\Absval{ \EE_{\nu_{N}} \bigl[ Z_{N}^{\sample} \bigr] - Z }
		\leq \EE_{\nu_{N}}\bigl[ \bigabsval{ Z_{N}^{\sample} - Z } \bigr]
		\leq 2 e^{C_{0}} \Norm{\EE_{\nu_{N}}\bigl[ \absval{ \Misfit - \Misfit_{N} } \bigr]}_{L^1_{\mu_{0}}}
		\leq \min \biggl\{ Z-\frac{1}{C_3},C_3-Z \biggr\}.
	\]
	The last inequality implies that assumption (\ref{item:hell_conv_marg_c}) of Theorem~\ref{thm:hell_conv_marg} holds with the same $C_3$ as in \eqref{eq:boundedness_of_EmuEnu_N_Phi_N_errorsequence}, since for any $0<C_3<+\infty$ that satisfies $C_3^{-1}<Z<C_3$ and \eqref{eq:boundedness_of_EmuEnu_N_Phi_N_errorsequence}, we have
	\begin{equation*}
	 	C_3^{-1}-Z\leq \EE_{\nu_{N}} \bigl[ Z_{N}^{\sample} \bigr] - Z \leq Z-C_3^{-1} \implies C_3^{-1}\leq \EE_{\nu_{N}} \bigl[ Z_{N}^{\sample} \bigr]
	\end{equation*}
	and 
	\begin{equation*}
		Z-C_3\leq \EE_{\nu_{N}} \bigl[ Z_{N}^{\sample} \bigr] - Z \leq C_3-Z \implies \EE_{\nu_{N}} \bigl[ Z_{N}^{\sample} \bigr] \leq C_3,
	\end{equation*}
	and combining both the implied statements yields assumption (\ref{item:hell_conv_marg_c}) of Theorem~\ref{thm:hell_conv_marg}; thus \eqref{eq:hell_conv_marg_result} holds, as desired.
	
	Now note that \eqref{eq:lower_bounds_on_potentials_Phi_PhiN} implies that assumption (\ref{item:hell_conv_rand_a}) of Theorem~\ref{thm:hell_conv_rand} holds with $q_1=q_2=+\infty$ and $D_1=4\exp(C_0)$.
	Furthermore, \eqref{eq:lower_bounds_on_potentials_Phi_PhiN} also implies that $Z_{N}^{\sample}=\EE_{\mu_{0}}[\exp(-\Misfit_{N})]\leq \exp(C_0)$ for all $\Misfit_{N}$.
	Thus, given that $Z$ is $\nu_{N}$-a.s.\ constant, and given that there exists some $0<C_3<\infty$ such that $C_3^{-1}<Z<C_3$,
	\begin{align}
		&\EE_{\nu_{N}} \left[ \bigl( Z_{N}^{\sample} \bigr)^{q_1} \max \bigl\{ Z^{-3}, \bigl( Z_{N}^{\sample} \bigr)^{-3} \bigr\}^{q_1} \big(\exp( -\Misfit(u) ) + \exp ( - \Misfit_{N}(u) ) \big)^{2q_1} \right]^{1/q_1}
		\notag
		\\
		& \quad \leq 4\exp(3C_0)\EE_{\nu_{N}} \left[ \max \bigl\{ C_3^{-3}, \bigl( Z_{N}^{\sample} \bigr)^{-3} \bigr\}^{q_1}\right]^{1/q_1}.
		\label{eq:prelim02}
	\end{align}
	A necessary and sufficient condition for setting $q_1=+\infty$ above (and therefore also in assumption (\ref{item:hell_conv_rand_b}) of Theorem~\ref{thm:hell_conv_rand}) is that $Z_{N}^{\sample}$ is $\nu_{N}$-a.s.\ bounded away from zero by a constant that does not depend on $N$.
	By the convexity and monotonicity of $x \mapsto \exp(x)$,
	\begin{equation*}
		Z_{N}^{\sample} = \EE_{\mu_{0}}\left[\exp(-\Misfit_{N})\right]\geq \exp \left(\EE_{\mu_{0}}\left[-\Misfit_{N}\right]\right)\geq \exp(-C_4),
	\end{equation*}
	for $C_4$ as in \eqref{eq:means_of_PhiN_uniformly_bounded}.
	In particular, if \eqref{eq:means_of_PhiN_uniformly_bounded} holds, then so does assumption (\ref{item:hell_conv_rand_b}) of Theorem~\ref{thm:hell_conv_rand}, with $q_1=q_2=+\infty$ and $D_2=4\exp(3C_0)\max\{C_3^{-3},\exp(3C_4)\}$, by inequality \eqref{eq:prelim02}.
\end{proof}

\begin{proof}[Proof of Lemma~\ref{lemma:goodsuffcon_for_hell_conv_thms_for_asmp_set_2}]	
	The proof proceeds in the same way as the proof of Lemma~\ref{lemma:goodsuffcon_for_hell_conv_thms_for_asmp_set_1}, with the exception that we need to prove that the assumption that $\EE_{\nu_{N}}[\exp(\rho^{\ast}\Misfit_{N})]\in L^1_{\mu_{0}}$ for some $\rho^{\ast}> 2$ implies that assumption (\ref{item:hell_conv_marg_a}) of Theorem~\ref{thm:hell_conv_marg} and assumption (\ref{item:hell_conv_rand_b}) of Theorem~\ref{thm:hell_conv_rand} hold with the stated parameters.
	Therefore, the proof will only concern these two assertions.
	Since $x\mapsto x^{-t}$ is strictly convex on $\mathbb{R}_{>0}$ for any $t>0$, Jensen's inequality yields that $\norm{ \EE_{\nu_{N}}[\exp(-\Misfit_{N})]^{-1} }_{L^t_{\mu_{0}}} \leq \norm{ \EE_{\nu_{N}}[\exp(t\Misfit_{N})] }^{1/t}_{L^1_{\mu_{0}}}$.
	Therefore, setting $t=\rho^{\ast}$, we find that assumption (\ref{item:hell_conv_marg_a}) of Theorem~\ref{thm:hell_conv_marg} holds, with $p_1 = \rho^{\ast}$ and $C_1 = \norm{ \EE_{\nu_{N}}[\exp(\rho^{\ast}\Misfit_{N})] }^{1/\rho^{\ast}}_{L^1_{\mu_{0}}}$.
	The inequality $\max\{x,y\}\leq x+y$ for $x,y\geq 0$ implies that
	\begin{align*}
		&\EE_{\nu_{N}} \left[ \max \bigl\{ Z_{N}^{\sample} Z^{-3}, \bigl( Z_{N}^{\sample} \bigr)^{-2} \bigr\}^{q_1} \big(\exp( -\Misfit(u) ) + \exp ( - \Misfit_{N}(u) ) \big)^{2q_1} \right]^{1/q_1}
		\\
		&\leq 4\exp(2C_0)\left(C_3^{-3}\exp(C_0)+\EE_{\nu_{N}}\left[\left(Z^{\sample}_N\right)^{-2q_1}\right]^{1/q_1}\right),
	\end{align*}
	while Jensen's inequality, Tonelli's theorem, and the definition of the $L^1_{\mu_{0}}$-norm yield that 
	\begin{align*}
		\EE_{\nu_{N}}[(Z^{\sample}_N)^{-2q_1}]
		& \leq \EE_{\nu_{N}} \left[ \EE_{\mu_{0}}\left[\exp(2q_1\Misfit_{N}) \right] \right] \\
		& = \EE_{\mu_{0}} \left[ \EE_{\nu_{N}} \left[ \exp(2q_1\Misfit_{N}) \right] \right] \\
		& = \Norm{\EE_{\nu_{N}} \left[ \exp(2q_1\Misfit_{N}) \right]}_{L^1_{\mu_{0}}}.
	\end{align*}
	Since the last term is finite for $q_1\leq \rho^\ast/2$ by the hypothesis that $\EE_{\nu_{N}}[\exp(\rho^{\ast}\Misfit_{N})]\in L^1_{\mu_{0}}$, it follows that
	assumption (\ref{item:hell_conv_rand_b}) of Theorem~\ref{thm:hell_conv_rand} holds with the parameters $q_1= \rho^{\ast}/2$, $q_2=+\infty$, and the scalar $D_2=4\exp(2C_0)(C_3^{-3}\exp(C_0)+\norm{ \EE_{\nu_{N}}[\exp(\rho^{\ast} \Misfit_{N})] }^{2/\rho^{\ast}}_{L^1_{\mu_{0}}})$.
\end{proof}

\begin{proof}[Proof of Proposition~\ref{proposition:bound_on_misfiterror_intermsof_forwardmodelerror}]
	Recall \eqref{eq:quadratic_misfits}, and fix an arbitrary $u\in\UU$.
	We have
	\begin{align}
		\bigabsval{ \Misfit(u) - \Misfit_{N}(u) }
		&= \frac{1}{2} \Absval{ \Innerprod{G(u)-y)}{\Gamma^{-1}(G(u)-y)}-\Innerprod{G_{N}(u)-y}{\Gamma^{-1}(G_{N}(u)-y)} }\notag.
	\end{align}
	Adding and subtracting $\Innerprod{G_{N}(u)-y)}{\Gamma^{-1}(G(u)-y)}$ inside the absolute value, rearranging terms, applying the Cauchy--Schwarz inequality, and letting $C_\Gamma$ be the largest eigenvalue of $\Gamma^{-1}$ yields
	\begin{align}
		\bigabsval{ \Misfit(u) - \Misfit_{N}(u) }
		&= \frac{1}{2} \Absval{ \Innerprod{\Gamma^{-1}(G(u)-y)}{G(u)-G_{N}(u)}
		+\Innerprod{\Gamma^{-1}(G_{N}(u)-y)}{G(u)-G_{N}(u)} } \notag
		\\
		& = \frac{1}{2} \Absval{ \Innerprod{ G(u)-y + G_{N}(u) - y }{\Gamma^{-1} (G(u) - G_{N}(u)) } } \notag
		\\
		&\leq C_\Gamma \norm{G(u)+G_{N}(u)-2y} \norm{G(u)-G_{N}(u)}.
		\label{eq:prelim01}
	\end{align}
	By the triangle inequality,
	\begin{align*}
		\norm{G(u)+G_{N}(u)-2y} \leq 2\max\{\norm{G(u)-y},\norm{G_{N}(u)-y}\} = 2\max\{\Misfit(u)^{1/2},\Misfit_{N}(u)^{1/2}\},
	\end{align*}
	and the triangle inequality and \eqref{eq:quadratic_misfits} yield
	\begin{align*}
		\Misfit_{N}(u)^{1/2}
		& =2^{-1/2} \norm{G_{N}(u)-y} \\
		& =2^{-1/2} \norm{G(u)-y+G_{N}(u)-G(u)} \\
		& \leq 2^{-1/2}(2^{1/2}\Misfit(u)^{1/2}+\norm{ G_{N}(u) - G(u) }) \\
		& =\Misfit(u)^{1/2}+2^{-1/2} \norm{ G_{N}(u) - G(u) } .
	\end{align*}
	Together, these inequalities yield
	\begin{align*}
		\norm{G(u)-y+G_{N}(u)-y}&\leq 2(\Misfit(u)^{1/2}+2^{-1/2} \norm{ G_{N}(u) - G(u) }),
	\end{align*}
	and substituting the above into \eqref{eq:prelim01} yields
	\begin{align*}
		\bigabsval{ \Misfit(u) - \Misfit_{N}(u) }\leq 2C_{\Gamma}\left(\Misfit(u)^{1/2} \norm{ G_{N}(u) - G(u) }+\norm{G(u)-G_{N}(u)}^{2}\right),
	\end{align*}
	thus proving  \eqref{eq:ae_in_u_bound_absval_error_potentials_01}.
	Using \eqref{eq:gen_triangle_inequality} yields
	\begin{equation*}
		\bigabsval{\Misfit(u)-\Misfit_{N}(u)}^q\leq 2^{q-1}(2C_\Gamma)^q\left(\Misfit(u)^{q/2} \norm{ G_{N}(u) - G(u) }^q+\norm{G(u)-G_{N}(u)}^{2q}\right).
	\end{equation*}
	Now take expectations with respect to $\nu_{N}$:
	since $G$ and $\Misfit$ are constant with respect to $\nu_{N}$,
	\[
		\EE_{\nu_{N}} \bigl[ \bigabsval{\Misfit(u)-\Misfit_{N}(u)}^q \bigr]
		\leq (4C_\Gamma)^q \Bigl( \Misfit(u)^{q/2}\EE_{\nu_{N}} \bigl[ \norm{ G_{N}(u) - G(u) }^q \bigr] + \EE_{\nu_{N}} \bigl[\norm{G(u)-G_{N}(u)}^{2q} \bigr] \Bigr) ,
	\]
	and taking the $q$\textsuperscript{th} root of both sides proves \eqref{eq:ae_in_u_bound_absval_error_potentials_01b}.
\end{proof}

\begin{proof}[Proof of Corollary~\ref{corollary:bound_on_misfiterror_intermsof_forwardmodelerror}]
	Taking the $L^s_{\mu_{0}}$ norm of both sides of the second inequality in Proposition~\ref{proposition:bound_on_misfiterror_intermsof_forwardmodelerror}, and applying \eqref{eq:gen_triangle_inequality} with $s/q\geq 1$, we obtain
	\begin{align*}
		&\bignorm{ \EE_{\nu_{N}}\left[\absval{ \Misfit - \Misfit_{N} }^q\right]^{1/q} }_{L^s_{\mu_{0}}}
		\\
		&\leq (4C_\Gamma)\EE_{\mu_{0}}\left[\left(\Misfit^{q/2}\EE_{\nu_{N}}\left[\norm{ G_{N} - G }^q\right]+\EE_{\nu_{N}}\left[\norm{ G_{N} - G }^{2q}\right]\right)^{s/q}\right]^{1/s}
		\\
		&\leq (4C_\Gamma)2^{1/q-1/s}\left(\EE_{\mu_{0}}\left[\Misfit(u)^{s/2}\EE_{\nu_{N}}\left[\norm{ G_{N} - G }^q\right]^{s/q}\right]+\EE_{\mu_{0}}\left[\EE_{\nu_{N}}\left[\norm{ G_{N} - G }^{2q}\right]^{s/q}\right]\right)^{1/s}.
	\end{align*}
	By the Cauchy--Schwarz inequality and Jensen's inequality,
	\begin{align*}
		\EE_{\mu_{0}}\left[\Misfit^{s/2}\EE_{\nu_{N}}\left[\norm{ G_{N} - G }^{q}\right]^{s/q}\right]&\leq \left(\EE_{\mu_{0}}\left[\Misfit^{s}\right]\EE_{\mu_{0}}\left[\EE_{\nu_{N}}\left[\norm{ G_{N} - G }^{q}\right]^{2s/q}\right]\right)^{1/2}
		\\
		&\leq\left(\EE_{\mu_{0}}\left[\Misfit^{s}\right]\EE_{\mu_{0}}\left[\EE_{\nu_{N}}\left[\norm{ G_{N} - G }^{2q}\right]^{s/q}\right]\right)^{1/2}.
	\end{align*}
	Since $0\leq a\leq 1 \implies a\leq a^{1/2}$, the hypotheses of the corollary and the preceding imply that
	\begin{align*}
		\bignorm{ \EE_{\nu_{N}}\left[\absval{ \Misfit - \Misfit_{N} }^q\right]^{1/q} }_{L^s_{\mu_{0}}} \leq (4C_\Gamma)2^{1/q - 1/s}\left(\EE_{\mu_{0}}\left[\Misfit^{s}\right]^{1/2}+1\right)^{1/s}
		\bignorm{ \EE_{\nu_{N}}\left[ \norm{ G_{N}-G }^{2q} \right]^{1/q} }_{L^s_{\mu_{0}}}^{1/2}.
	\end{align*}
	Since $2^{1/q-1/s}\leq 2^{1/q}\leq 2$, the proof is complete.
\end{proof}

\begin{proof}[Proof of Lemma~\ref{lemma:necessary_condition_of_convergence_of_forward_models}]
	Given \eqref{eq:quadratic_misfits}, we may choose the parameter $C_0$ in \eqref{eq:lower_bounds_on_potentials_Phi_PhiN} to be $C_0=0$.
	By Jensen's inequality, \eqref{eq:convergence_hypothesis_of_forward_models} implies \eqref{eq:boundedness_of_EmuEnu_N_Phi_N_errorsequence}.
\end{proof}

\begin{proof}[Proof of Theorem~\ref{thm:cvgce_of_random_measures}]
	We first verify that Assumption~\ref{asmp:goodsuffcon_for_hellconvthms} holds.
	Since $\Misfit$ and $\Misfit_{N}$ satisfy \eqref{eq:quadratic_misfits}, it follows that we may set $C_0=0$ in \eqref{eq:lower_bounds_on_potentials_Phi_PhiN}.
	Since we assume throughout that $0<Z=\EE_{\mu_{0}}[\exp(-\Misfit)]<\infty$, it follows that $\Misfit$ has moments of all orders, and hence belongs to $L^s_{\mu_{0}}$ for all $s\in\Naturals$.
	Therefore, given that \eqref{eq:convergence_hypothesis_of_forward_models} holds for $q,s\geq 1$, it follows from Jensen's inequality and Corollary~\ref{corollary:bound_on_misfiterror_intermsof_forwardmodelerror} that we can make $\norm{\EE_{\nu_{N}} [ \absval{\Misfit_{N}-\Misfit} ] }_{L^1_{\mu_{0}}}$ as small as desired.
	In particular, for any $0<C_3<+\infty$ that satisfies $C_3^{-1}<Z<C_3$, there exists a $N^\ast(C_3) \in \Naturals$ such that, for all $N \geq N^\ast(C_3)$, \eqref{eq:boundedness_of_EmuEnu_N_Phi_N_errorsequence} holds.

	The rest of the proof consists of applying Lemma~\ref{lemma:goodsuffcon_for_hell_conv_thms_for_asmp_set_1} or \ref{lemma:goodsuffcon_for_hell_conv_thms_for_asmp_set_2}, Corollary~\ref{corollary:bound_on_misfiterror_intermsof_forwardmodelerror} and Lemma~\ref{lemma:necessary_condition_of_convergence_of_forward_models}.

	\emph{Case (\ref{thm:cvgce_of_random_measures_1})}.
	The hypotheses in this case ensure that we may apply Lemma~\ref{lemma:goodsuffcon_for_hell_conv_thms_for_asmp_set_1}.
	Set $p_1=p^\ast$ and $p_2=p_3=+\infty$, so that $p_1'=(p^\ast)'=p^\ast/(p^\ast-1)$ and $p_2'=p_3'=1$.
	Substituting these exponents into \eqref{eq:hell_conv_marg_result_inequality} and applying Corollary~\ref{corollary:bound_on_misfiterror_intermsof_forwardmodelerror} with $s=2p_1'p_3'=2p^\ast/(p^\ast-1)$ and $q=p_2'=1$ (note that $s\geq q\geq 1$), we obtain
	\begin{align*}
		\dhh \bigl( \mu, \mu_{N}^{\marginal} \bigr) &\leq C \Norm{ \EE_{\nu_{N}} \bigl[ \absval{ \Misfit - \Misfit_{N} }\bigr] }_{L^{2p^\ast/(p^\ast-1)}_{\mu_{0}}(\UU)}\leq C \Norm{ \EE_{\nu_{N}} \left[ \norm{ G - G_{N} }^{2} \right] }_{L^{2p^\ast/(p^\ast-1)}_{\mu_{0}}(\UU)}^{1/2},
	\end{align*}
	where $C>0$ changes value between inequalities.
	Thus we have shown that \eqref{eq:hell_conv_marg_result_forward_model} holds with $r_1=1$ and $r_2=2p^\ast/(p^\ast-1)$.

	To prove that \eqref{eq:hell_conv_rand_result_forward_model} holds with the desired exponents, we again use Lemma~\ref{lemma:goodsuffcon_for_hell_conv_thms_for_asmp_set_1} to set $q_1=q_2=+\infty$, so that $q_1'=q_2'=1$.
	Substituting these exponents into \eqref{eq:hell_conv_rand_result}, and applying Corollary~\ref{corollary:bound_on_misfiterror_intermsof_forwardmodelerror} with $s=2q_2'=2$ and $q=2q_1'=2$, we obtain
	\begin{align*}
		\EE_{\nu_{N}} \left[ \dhh \bigl( \mu, \mu_{N}^{\sample} \bigr)^{2} \right]^{1/2} \leq D \Norm{ \EE_{\nu_{N}} \bigl[ \absval{ \Misfit - \Misfit_{N} }^{2} \bigr]^{1/2}}_{L^{2}_{\mu_{0}}}\leq D \Norm{ \EE_{\nu_{N}} \left[ \norm{ G - G_{N} }^{4} \right]^{1/2} }_{L^{2}_{\mu_{0}}(\UU)}^{1/2},
	\end{align*}
	where $D>0$ changes value between inequalities.
	Thus we have shown that \eqref{eq:hell_conv_rand_result_forward_model} holds with $s_1=s_2=2$.

	It remains to ensure that both the rightmost terms above converge to zero.
	Since \eqref{eq:convergence_hypothesis_of_forward_models} holds with $q=2$ and $s=2p^\ast/(p^\ast-1)$, the desired convergence follows from the nesting property of finite-measure $L^p$-spaces.
	Therefore, both $\mu^{\marginal}_N$ and $\mu^{\sample}_N$ converge to $\mu$ as claimed.

	\emph{Case (\ref{thm:cvgce_of_random_measures_2})}.
	Since the arguments in this case are the same as in the previous case, we only record the different material.

	The hypotheses ensure that we may apply Lemma~\ref{lemma:goodsuffcon_for_hell_conv_thms_for_asmp_set_2}.
	Set $p_1=\rho^{\ast}$ and $p_2=p_3=+\infty$, so that $(p_1)'=\rho^{\ast}/(\rho^{\ast}-1)$ and $p_2'=p_3'=1$.
	Substituting these exponents into \eqref{eq:hell_conv_marg_result_inequality} and applying Corollary~\ref{corollary:bound_on_misfiterror_intermsof_forwardmodelerror} with $s=2p_1'p_3'=2\rho^{\ast}/(\rho^{\ast}-1)$ and $q=p_2'=1$, we obtain 
	\begin{align*}
		\dhh \bigl( \mu, \mu_{N}^{\marginal} \bigr) &\leq C \Norm{ \EE_{\nu_{N}} \bigl[ \absval{ \Misfit - \Misfit_{N} }\bigr] }_{L^{2\rho^{\ast}/(\rho^{\ast}-1)}_{\mu_{0}}(\UU)}\leq C \Norm{ \EE_{\nu_{N}} \left[ \norm{ G - G_{N} }^{2} \right] }_{L^{2p^\ast/(p^\ast-1)}_{\mu_{0}}(\UU)}^{1/2},
	\end{align*}
	where $C>0$ changes value between inequalities.
	Thus we have shown that \eqref{eq:hell_conv_marg_result_forward_model} holds with $r_1=1$ and $r_2=2\rho^{\ast}/(\rho^{\ast}-1)$.

	To prove that \eqref{eq:hell_conv_rand_result_forward_model} holds with the desired exponents, we again use Lemma~\ref{lemma:goodsuffcon_for_hell_conv_thms_for_asmp_set_2} to set $q_1=\tfrac{\rho^{\ast}}{2}$ and $q_2=+\infty$, so that $q_1'=\rho^{\ast}/(\rho^{\ast}-2)$ and $q_2'=1$.
	Substituting these exponents into \eqref{eq:hell_conv_rand_result}, and applying Corollary~\ref{corollary:bound_on_misfiterror_intermsof_forwardmodelerror} with $s=2q_2'=2$ and $q=2q_1'=2\rho^{\ast}/(\rho^{\ast}-2)$, we obtain
	\begin{align*}
		\EE_{\nu_{N}} \left[ \dhh \bigl( \mu, \mu_{N}^{\sample} \bigr)^{2} \right]^{1/2} &\leq D \Norm{ \EE_{\nu_{N}} \bigl[ \absval{ \Misfit - \Misfit_{N} }^{2\rho^{\ast}/(\rho^{\ast}-2)} \bigr]^{(\rho^{\ast}-2)/(2\rho^{\ast})}}_{L^{2}_{\mu_{0}}}
		\\
		&\leq D\Norm{ \EE_{\nu_{N}} \bigl[ \norm{ G - G_{N} }^{4\rho^{\ast}/(\rho^{\ast}-2)} \bigr]^{(\rho^{\ast}-2)/(2\rho^{\ast})} }_{L^{2}_{\mu_{0}}(\UU)}^{1/2},
	\end{align*}
	where $D>0$ changes value between inequalities.
	Thus \eqref{eq:hell_conv_rand_result_forward_model} holds with $s_1=2\rho^{\ast}/(\rho^{\ast}-2)$ and $s_2=2$.
	Since \eqref{eq:convergence_hypothesis_of_forward_models} holds with $q=2\rho^{\ast}/(\rho^{\ast}-2)$ and $s=2\rho^{\ast}/(\rho^{\ast}-1)$, it follows from the nesting property of $L^p$-spaces defined on finite measure spaces that both
	\[
		\Norm{ \EE_{\nu_{N}} \left[ \norm{ G - G_{N} }^{2} \right] }_{L^{2p^\ast/(p^\ast-1)}_{\mu_{0}}(\UU)}^{1/2}
		\text{ and }
		\Norm{ \EE_{\nu_{N}} \left[ \norm{ G - G_{N} }^{4\rho^{\ast}/(\rho^{\ast}-2)} \right]^{(\rho^{\ast}-2)/(2\rho^{\ast})} }_{L^{2}_{\mu_{0}}(\UU)}^{1/2}
	\]
	converge to zero.
\end{proof}

\begin{proof}[Proof of Proposition~\ref{prop:ell_sparse_conv}] 
	We start by verifying the assumptions of Theorem~\ref{thm:hell_conv_rand}.
	First, since $\Misfit(u) \geq 0$ for all $u \in \UU$, and $\Misfit_{N}(u) \geq 0$ for all $u \in \UU$ and all $\{\sigma^{(i)}\}_{i=1}^N$, assumption (\ref{item:hell_conv_rand_a}) is satisfied for $q_1=q_2=\infty$.
	For assumption (\ref{item:hell_conv_rand_b}), we then have, for any $q_2\in [1,\infty]$,
	\begin{align*}
		&\Norm{ \Big(\EE_{\sigma} \left[ \bigl( Z_{N}^{\sample} \bigr)^{q_1} \max \{ Z^{-3}, \bigl( Z_{N}^{\sample} \bigr)^{-3} \}^{q_1} \big(\exp \big(-\Misfit(u)\big) + \exp \big(-\Misfit_{N}(u)\big)\big)^{2q_1} \right]^{1/q_1} }_{L^{q_2}_{\mu_{0}}(\UU)} \\
		& \quad \leq 4 \, \EE_{\sigma} \left[ \bigl( Z_{N}^{\sample} \bigr)^{q_1} \max \{ Z^{-3}, \bigl( Z_{N}^{\sample} \bigr)^{-3} \}^{q_1} \right]^{1/q_1} \\
		& \quad \leq 4 \left( Z^{-3q_1} \EE_{\sigma} \bigl[ \bigl( Z_{N}^{\sample} \bigr)^{q_1} \bigr] + \EE_{\sigma} \bigl[ \bigl( Z_{N}^{\sample} \bigr)^{-2q_1} \bigr] \right)^{1/q_1}.
	\end{align*}
	Since $\Misfit_{N}(u) \geq 0$ for all $u \in \UU$ and all $\{\sigma^{(i)}\}_{i=1}^N$, we have for any $q_1\in [1,\infty]$
	\begin{equation}\label{eq:proof1}
		\EE_{\sigma} \left[ \bigl( Z_{N}^{\sample} \bigr)^{q_1} \right]^{1/q_1} = \EE_{\sigma} \left[ \left(\int_{\UU} \exp(-\Misfit_{N}(u)) \, \rd \mu_{0} (u) \right)^{q_1} \right]^{1/q_1} \leq 1.
	\end{equation}
	Using the $\ell$-sparse distribution of $\sigma$, we further have $\absval{ \sigma^{(i)}_j } \leq \sqrt{s}$ and
	\begin{align*}
		\Misfit_{N}(u) = \frac{1}{2 N} \sum_{i = 1}^{N} \bigabsval{ {\sigma^{(i)}}^{\transpose} \bigl( \Gamma^{-1/2} ( y - G(u) ) \bigr) }^{2} 
		\leq \frac{s}{2} \bignorm{ \bigl( \Gamma^{-1/2} ( y - G(u) ) \bigr) }^{2} 
		= s \Misfit(u),
	\end{align*}
	which implies that $Z_{N}^{\sample} \geq Z_s = \int_\UU \exp(-s \Misfit(u)) \, \rd \mu_{0} (u)$.
	It follows that, for any $q_1\in [1,\infty]$,
	\[
		\EE_{\sigma} \bigl[ \bigl( Z_{N}^{\sample} \bigr)^{-2q_1} \bigr]^{1/q_1} \leq \EE_{\sigma} \left[ Z_s^{-2q_1} \right]^{1/q_1} = Z_s^{-2},
	\]
	and assumption (\ref{item:hell_conv_rand_b}) is hence also satisfied for $q_1=q_2=\infty$.
	Hence, by Theorem~\ref{thm:hell_conv_rand},
	\[
		\Bigl( \EE_{\sigma} \bigl[ \dhh \bigl( \mu, \mu_{N}^{\sample} \bigr)^{2} \bigr] \Bigr)^{1/2} \leq C \Norm{ \bigl( \EE_{\sigma} \bigl[ \absval{ \Misfit(u) - \Misfit_{N}(u) }^{2} \bigr] \bigr)^{1/2} }_{L^{2}_{\mu_{0}}(\UU)}.
	\]
	Using standard properties of Monte Carlo estimators (see e.g.\ \citet{RobertCasella:1999}), we have 
	\[
		\Bigl( \EE_{\sigma} \bigl[ \absval{ \Misfit(u) - \Misfit_{N}(u) }^{2} \bigr] \Bigr)^{1/2} = \sqrt{\frac{\Var_\sigma \bigr[ \frac{1}{2} \bigabsval{ \sigma^{\transpose} \Gamma^{-1/2} ( y - G(u) ) }^{2} \bigr]}{N}}.
	\]
	Now, using $\Var[X] = \EE[X^{2}] - \EE[X]^{2}$, $\left( \sum_{j=1}^{J} x_j \right)^4 \leq J^3 \sum_{j=1}^{J} x_j^4$, the linearity of expectation, the $\ell$-sparse distribution of $\sigma$, and $\norm{ x }_4 \leq \norm{ x }_2$, we have
	\begin{align*}
		0 
		& \leq \Var_\sigma \left[\frac{1}{2} \bigabsval{ \sigma^{\transpose} \Gamma^{-1/2} ( y - G(u) ) }^{2} \right] \\
		&= \EE_{\sigma} \left[ \frac{1}{4} \bigabsval{ \sigma^{\transpose} \Gamma^{-1/2} ( y - G(u) ) }^4 \right] - \EE_{\sigma} \left[ \frac{1}{4} \bigabsval{ \sigma^{\transpose} \Gamma^{-1/2} ( y - G(u) ) }^{2} \right]^{2} \\
		&= \EE_{\sigma} \left[ \frac{1}{4} \Absval{ \sum_{j=1}^{J} \sigma_j \bigl( \Gamma^{-1/2} ( y - G(u) ) \bigr)_j }^{4} \right] - \frac{1}{4} \bignorm{ \Gamma^{-1/2} ( y - G(u) ) }^{4} \\
		&\leq \frac{1}{4} J^3 \sum_{j=1}^{J} \EE_{\sigma} [\sigma_j^4] \bigl( \Gamma^{-1/2} ( y - G(u) ) \bigr)_j^4- \frac{1}{4} \bignorm{ \Gamma^{-1/2} ( y - G(u) ) }^{4} \\
		&= \frac{1}{4} J^3 \EE_{\sigma} [\sigma_j^4] \bignorm{ \Gamma^{-1/2} ( y - G(u) ) }_4^{4}- \frac{1}{4} \bignorm{ \Gamma^{-1/2} ( y - G(u) ) }^{4} \\
		&\leq \left( J^3 \EE_{\sigma} [\sigma_j^4] - 1\right) \Misfit(u)^{2}.
	\end{align*}
	The claim \eqref{eq:ell_sparse_conv_1} now follows, with the choice of constant as in \eqref{eq:ell_sparse_conv_2}.
\end{proof}


\begin{proof}[Proof of Theorem~\ref{thm:convergence_misfit_pn_ode}]
	Recall that $T_J$ is a set of time points in $[0,T]$, indexed by an index set $J$ with cardinality $\absval{ J } \in \Naturals$.
	In \eqref{eq:error_in_forward_operators}, we observed that
	\begin{equation*}
		\norm{ G_{N}(u) - G(u) } \leq \absval{ J } \sup_{0\leq t\leq T} \norm{ e(t; u) }_{\ell^d_2}.
	\end{equation*}
	Fix $\rho^{\ast}>2$.
	Omitting the argument $u$ of $\Misfit_{N}$, $\Misfit$, $G_{N}$ and $G$, we have
	\begin{align*}
		\exp \bigl(\rho^{\ast}\Misfit_{N}\bigr)
		& = \exp \bigl(\rho^{\ast}\bigl(\Misfit_{N}-\Misfit+\Misfit\bigr)\bigr) \\
		& \leq \exp \bigl(\rho^{\ast}\absval{\Misfit_{N}-\Misfit}+\rho^{\ast}\Misfit\bigr) \\
		& = \exp \bigl(\rho^{\ast}\absval{\Misfit_{N}-\Misfit}\bigr)\exp(\rho^{\ast}\Misfit)
		\\
		&\leq \exp \bigl(2\rho^{\ast} C_\Gamma\bigl(\Misfit^{1/2} \norm{ G_{N} - G }+\norm{G - G_{N}}^{2}\bigr)\bigr)\exp(\rho^{\ast}\Misfit)
		\\
		&\leq \frac{\exp(\rho^{\ast}\Misfit)}{2}\bigl[\exp \bigl(4\rho^{\ast} C_\Gamma \Misfit^{1/2}\norm{ G_{N} - G }\bigr)+\exp \bigl(4\rho^{\ast} C_\Gamma\norm{ G - G_{N} }^{2}\bigr)\bigr] ,
	\end{align*}
	where the last two inequalities follow from \eqref{eq:ae_in_u_bound_absval_error_potentials_01} and Young's inequality $a b \leq (a^{2} + b^{2}) / 2$ for $a, b \geq 0$.
	Using \eqref{eq:error_in_forward_operators}, we therefore obtain
	\begin{align*}
		\exp(\rho^{\ast}\Misfit_{N})
		& \leq 
		\frac{\exp(\rho^{\ast}\Misfit)}{2}\left[\exp \left(4\rho^{\ast} C_\Gamma \Misfit^{1/2} \absval{J}\sup_{0\leq t\leq T} \norm{e(t)}_{\ell^d_2}\right) \right. \\
		& \phantom{=} \quad \left. + \exp \left(4 \rho^{\ast} C_\Gamma\absval{J}^2\sup_{0\leq t\leq T} \norm{e(t)}^{2}_{\ell^d_2}\right)\right],
	\end{align*}
	where we note that we have suppressed the $u$-dependence of $e(t;u)$ and simply written $e(t)$.
	Since $\UU$ is compact and $S$ is continuous, it follows that $G$ and hence $\Misfit$ are continuous on $\UU$; by the extreme value theorem, $\Misfit$ is bounded on $\UU$, i.e.\ $\norm{ \Misfit }_{L^\infty_{\mu_{0}}(\UU)}$ is finite.
	Using this fact and taking expectations with respect to $\nu_{N}$ we obtain
	\begin{align*}
		\EE_{\nu_{N}}\left[\exp(\rho^{\ast}\Misfit_{N}(u))\right] \leq 
		\frac{\exp(\rho^{\ast}\norm{ \Misfit }_{L^\infty_{\mu_{0}}(\UU)})}{2} & \left(\EE_{\nu_{N}}\left[\exp \left(4\rho^{\ast} C_\Gamma \norm{ \Misfit }_{L^\infty_{\mu_{0}}(\UU)}^{1/2} \absval{J}\sup_{0\leq t\leq T} \norm{ e(t; u) }_{\ell^d_2}\right)\right]\right.
		\\
		&\quad\quad+\left.\EE_{\nu_{N}}\left[\exp \left(4 \rho^{\ast} C_\Gamma\absval{J}^2\sup_{0\leq t\leq T} \norm{ e(t; u) }^{2}_{\ell^d_2}\right)\right]\right).
	\end{align*}
	By Corollary~\ref{cor:finite_mgf_lipschitz_flow_continuous_time}, the two terms on the right-hand side are finite for every $u\in \UU$.
	Given the continuous dependence of the parameters of Assumptions~\ref{ass:exact_flow}, \ref{ass:numerical_flow}, and \ref{ass:p-R_regularity_condition_on_noise} on $u$, and given that $\UU$ is a compact subset of a finite-dimensional Euclidean space, it follows that the right-hand side can be bounded by a scalar that does not depend on any $u$.
	Hence, the function $u\mapsto\EE_{\nu_{N}}[\exp(\rho^{\ast}\Misfit_{N}(u))]$ belongs to $ L^\infty_{\mu_{0}}(\UU) \subset L^1_{\mu_0}(\UU)$, so that the first hypothesis of Theorem~\ref{thm:cvgce_of_random_measures}(\ref{thm:cvgce_of_random_measures_2}) holds.
	For the second hypothesis, observe that, since Assumption~\ref{ass:p-R_regularity_condition_on_noise} holds for $R=+\infty$, it follows that \eqref{eq:strong_error_sup_inside_continuous} holds for any $n\in\Naturals$, and thus \eqref{eq:convergence_hypothesis_of_forward_models} holds for any $q,s\geq 1$.
	Therefore the hypotheses of Theorem~\ref{thm:cvgce_of_random_measures}(\ref{thm:cvgce_of_random_measures_2}) are satisfied, and the desired conclusion follows from Theorem~\ref{thm:cvgce_of_random_measures}.
\end{proof}

\end{document}